\documentclass[10pt, a4paper]{amsart}
\usepackage{amsmath,amstext,amscd, amssymb,txfonts, epsfig, psfrag, color, multicol, graphicx}
\usepackage[normalem]{ulem}
\usepackage[colorlinks=true, pdfstartview=FitV, linkcolor=blue, citecolor=blue, urlcolor=blue]{hyperref}

\theoremstyle{plain}

\newtheorem{thm}{Theorem}[section]
\newtheorem{lemma}[thm]{Lemma}
\newtheorem{cor}[thm]{Corollary}
\newtheorem{prop}[thm]{Proposition}

\theoremstyle{definition}
\newtheorem{defn}[thm]{Definition}

\newtheorem{Open questions}[thm]{Open questions}
\newtheorem{Open question}[thm]{Open question}
\newtheorem{Open problems}[thm]{Open problems}
\newtheorem{Open problem}[thm]{Open problem}

\definecolor{dmagenta}{rgb}{.5,0,.5} 
\definecolor{dred}{rgb}{.5,0,0} 
\definecolor{dgreen}{rgb}{0,.5,0} 
\definecolor{dblue}{rgb}{0,0,0.5} 
\definecolor{black}{rgb}{0,0,0} 
\definecolor{vdgreen}{rgb}{0,.3,0} 
\definecolor{vdred}{rgb}{.3,0,0} 
\definecolor{red}{rgb}{1,0,0}

\def\Bbb{\mathbb}

\def\Z{\Bbb{Z}}
\def\N{\Bbb{N}}

\def\ni{\noindent}
\def\ie{{{\it i.e.}\/}\ }

\def\Area{\hbox{\rm Area}}

\def\Dist{\hbox{\rm Dist}}

\def\F+L{\hbox{$\textup{F}\!_+\textup{L}$}}

\renewcommand{\H}{\mathcal{H}}

\newcommand{\FreeEq}{\stackrel{\textup{fr}}{=}}

\def\ms{\medskip}

\def\onto{{\kern3pt\to\kern-8pt\to\kern3pt}}

\def\<{\langle}
\def\>{\rangle}
\def\|{{\ |\ }}

\newcommand{\set}[1]{\left\{#1\right\}}

\newcommand{\abs}[1]{\left|#1\right|}
\renewcommand{\ni}{\noindent}

\renewcommand{\ms}{\medskip}

\def\*{^{\star}}

\begin{document}

\title{Hydra Groups}

\author{W.\ Dison and T.R.\ Riley}

\date \today

\begin{abstract}
\ni We give examples of  $\textup{CAT}(0)$, biautomatic, free--by--cyclic, one--relator groups which have finite--rank free subgroups of huge (Ackermannian) distortion.  This leads to elementary examples of groups whose Dehn functions are similarly extravagant.  This behaviour originates in manifestations  of \emph{Hercules--versus--the--hydra battles}  in string--rewriting.

 \ms

\footnotesize{\ni \textbf{2010 Mathematics Subject
Classification:  20F65, 20F10, 20F67}  \\ \ni \emph{Key words and phrases:} free--by--cyclic group, subgroup distortion, Dehn function, hydra}
\end{abstract}

\maketitle

 \section{Introduction}

\subsection{Hercules versus the hydra}

Hercules' second labour was to fight the Lernaean hydra, a beast with multiple serpentine heads enjoying magical regenerative powers: whenever a head was severed, two grew in its place.   Hercules succeeded with the help of his nephew, Iolaus, who stopped the regrowth by searing the stumps with a burning torch after each decapitation.  The extraordinarily fast--growing functions we will encounter in this article stem from a re-imagining of this battle.

For us, a \emph{hydra} will be a finite--length \emph{positive} word on the alphabet $a_1, a_2, a_3, \ldots$ --- that is, it includes no inverse letters ${a_1}^{-1}, {a_2}^{-1}, {a_3}^{-1}, \ldots$.    Hercules fights a hydra by striking off its first letter.  The hydra then regenerates as follows: each remaining letter $a_i$, where $i>1$, becomes $a_ia_{i-1}$ and the $a_1$ are unchanged.  This process --- removal of the first letter and then regeneration --- repeats, with Hercules victorious \emph{when} (not \emph{if}!) the hydra is reduced to the empty word $\varepsilon$.

For example, Hercules defeats the hydra $a_2 a_3 a_1$ in five strikes:
$$a_2 a_3 a_1 \ \to \   a_3 a_2  a_1 \ \to \   a_2 a_1 a_1 \  \to \   a_1 a_1 \  \to \  a_1 \  \to \  \varepsilon.$$
(Each arrow represents the removal of the first letter and then regeneration.)

\begin{prop} \label{Hercules always wins}
Hercules defeats all hydra.
\end{prop}

\begin{proof}
When fighting a hydra in which the highest index present is $k$,  no $a_i$ with $i >k$ will ever appear, and nor will any new $a_k$.   The prefix before the first $a_k$ is itself a hydra, which, by induction, we can assume Hercules defeats.   Hercules will then remove that $a_k$, decreasing the total number of $a_k$ present.  It follows that Hercules eventually wins.
 \end{proof}

However these battles are of extreme duration.  Define $\H(w)$ to be the number of strikes it takes Hercules to vanquish the hydra $w$, and for integers $k \geq 1$, $n \geq 0$,  define $\H_k(n) := \H({a_k}^n)$.
We call the $\H_k$ \emph{hydra functions}.  Here are some values of $\H_k(n)$.
$$\begin{array}{r | c c c c c c c }
                           & 1 & 2 & 3 & 4 &  \cdots  & n &  \cdots   \\ \hline
1                         & 1 & 2 & 3 & 4 &  \cdots  & n &  \cdots \\
2                         & 1 & 3 & 7 & 15 &  \cdots  & 2^n-1 &  \cdots \\
3                         & 1 & 4 & 46 & 3.2^{46}-2 
&  \cdots  & \cdots &  \cdots \\ \vspace{-2mm}
\vdots                & \vdots & \vdots & \vdots  & \vdots  &     &   &   \\   \vspace{-2mm}
k                         & 1 & k+1 & \vdots & \vdots &    &   &    \\ \vspace{-2mm}
\vdots                & \vdots & \vdots &   \vdots &  \vdots &     &   &   \\
\end{array} $$

To see that $\H_2(n) = 2^n-1$ for all $n$, note that $$\H\left({a_2}^{n+1}\right) \ =  \  \H\left({a_2}^{n}\right) +  \H\left(a_2{a_1}^{\H({a_2}^{n})}\right) \ = \   2\H\left({a_2}^{n}\right)+1.$$
 And  $\H_3(n)$ is essentially an $n$--fold iterated exponential function because, for all $n >0$, $$\H_3(n+1) \  = \  3. 2^{\mbox{$ \H_3(n)$}} -2,$$ by the calculations
\begin{eqnarray*}
  \H( {a_3}^{n+1} )  & = &   \H({a_3}^{n}) +1 + \H \left({a_2}{a_1} \, {a_2}{a_1}^2 \, \ldots \, {a_2}{a_1}^{\mbox{$ \H({a_3}^{n})$}}\right), \\
 \H({a_2}{a_1} \, {a_2}{a_1}^2 \, \ldots \, {a_2}{a_1}^m) & = & 3 . 2^m -m -3.
 \end{eqnarray*}

%%In particular, $\H_3(4) = 3.2^{46}-2  =  211106232532990$.

Extending this line of reasoning, we will derive relationships \eqref{eq4} and \eqref{phi recursion relation} in Section~\ref{hydra functions versus Ackermann's functions} from which it will follow that  $$\H_4(3)  \ =  \   3 \cdot 2^{3 \cdot 2^{3 \cdot 2^{3 \cdot 2^{5}-1}-1}-1}-1.$$ 
So these functions are extremely wild.   The reason behind the fast growth is a nested recursion.  What we have is a variation on Ackermann's functions $A_k  : \N \to \N$,  defined for integers $k,n \geq 0$ by:
\begin{align*}
A_0(n) & \ = \ n+2    \textup{ for }    n \geq 0,   \\
A_k(0) & \ =  \
\begin{cases}
0 & \textup{ for }   k = 1  \\
1 & \textup{ for }   k \geq 2,
\end{cases}
     \\
  \ \ \ \text{ and }  \ \ \
A_{k+1}(n+1) & \  =  \  A_k({A_{k+1}}(n))  \textup{ for }   k,n  \geq 0.
\end{align*}
So, in particular, $A_1(n)=2n$,  $A_2(n)=2^n$ and $A_3(n) = \exp_2^{(n)}(1)$, the $n$--fold iterated power of $2$.
(Definitions of Ackermann's functions occur with minor variations in the literature.) 
Ackermann's functions are represent the successive levels of the Grzegorczyk hierarchy, which is a grading of all primitive recursive functions --- see, for example, \cite{Rose}.

We will prove the following relationship in Section~\ref{hydra functions versus Ackermann's functions}.  Our notation in this proposition and henceforth is that for $f,g:\N \to \N$, we write $f \preceq g$ when there exists $C>0$ such that for all $n$ we have $f(n) \leq Cg(Cn+C)+Cn+C$.  This gives an equivalence relation capturing qualitative agreement of growth rates:  $f \simeq g$ if and only if $f \preceq g$ and $g \preceq f$.

\begin{prop} \label{prop2}
For all $k \geq 1$, $\H_k \simeq A_k$.
\end{prop}

Other hydra dwell in the mathematical literature, particularly in the context of results concerning independence  from  Peano arithmetic and other logical systems.    The hydra of Kirby and Paris \cite{Kirby-Paris}, based on finite rooted trees, are particularly celebrated.   Similar, but yet more extreme hydra were later constructed by Buchholz \cite{Buchholz}.  And creatures that, like ours, are finite strings that  regenerate on \emph{decapitation}  were defined by Hamano and Okada \cite{Hamano-Okada} and then independently by Beklemishev~\cite{Beklemishev}. They go by the name of \emph{worms}, are descended from Buchholz's hydra, involve more complex regeneration rules, and withstand Hercules even longer.

\subsection{Wild subgroup distortion} \label{wild distortion}

The distortion function $\Dist^{G}_{H} : \N \to \N$ for a subgroup $H$ with finite generating set $T$ inside a group $G$ with finite generating set $S$ compares the intrinsic word metric $d_T$ on $H$ with the extrinsic word metric $d_S$: 
$$\Dist^{G}_{H}(n) \ := \  \max \set{  \  d_T(1,g)  \  \mid  \   g \in {H} \textup{ with }  d_S(1,g)  \leq n   \ }.$$  Up to $\simeq$ it is does not depend on the particular finite generating sets used.

A manifestation of our Hercules--versus--the--hydra battle leads to the result that even for apparently benign $G$ and $H$, distortion can be wild.

 \begin{thm} \label{main}
For each integer $k \geq 1$, there is a finitely generated group $G_k$ that
\begin{itemize}
\item is free--by--cyclic,
\item can be presented with only one defining relator,
\item is $\textup{CAT}(0)$,
\item is biautomatic,
\item and enjoys the rapid decay property,
\end{itemize}
and yet has a rank--$k$ free subgroup $H_k$ that is distorted like the $k$--th of Ackermann's functions --- that is, $\Dist^{G_k}_{H_k}  \simeq A_k$.
\end{thm}

This distortion of a free subgroup of a  $\textup{CAT}(0)$ group stands in stark contrast to that of any abelian subgroup --- they are always quasi--isometrically embedded  (see Theorem~4.10 of Chapter III.$\Gamma$ in \cite{BrH}, for example) and so no more than linearly distorted.
The distortion we achieve exceeds that found in the hyperbolic groups of Mitra~\cite{Mitra} and the subsequent 2-dimensional $\textup{CAT}(-1)$ groups  of Barnard, Brady and Dani \cite{BBD}: first, for all $k$, they give examples with  a free subgroup of  distortion $\simeq \exp^{(k)}(n)$, and then they give examples with free subgroups whose distortion functions grow faster than  $\exp^{(k)}(n)$ for every $k$.  However, our examples contain $\Z^2$ subgroups and so are not hyperbolic.

Explicitly, our examples are
\begin{equation} \label{our groups}
G_k \ = \ \langle \  a_1, \ldots, a_k,  t \  \mid \  t^{-1} {a_1} t=a_1, \ t^{-1}{a_i} t=a_i a_{i-1} \ (\forall i>1)   \ \rangle\end{equation}
and their subgroups $$H_k :=  \langle a_1t, \ldots, a_kt \rangle.$$
So $G_k$ is the free--by--cyclic group $F(a_1, \ldots, a_k) \rtimes \Z$ where $\Z= \langle t \rangle$ and $t$ acts by the automorphism  of $F(a_1, \ldots, a_k)$ that is the restriction of the automorphism $\theta$ of $F(a_1, a_2, \ldots)$
defined by
\begin{align}
    \theta(a_i) \ = \ \begin{cases}
      a_1 \quad & i=1, \\
      a_i a_{i-1} \quad & i > 1.
    \end{cases} \label{theta definition}
  \end{align}

For $i \leq j$, the canonical homomorphism $G_i \to G_j$ is an inclusion  as  the free--by-cyclic normal forms of an element of $G_i$ and its image in $G_j$ are the same.   So the direct limit of the $G_i$ under these inclusions is
$$G  \ = \ \langle \  t, a_1, a_2, \ldots    \  \mid \  t^{-1} {a_1} t=a_1, \ t^{-1}{a_i} t=a_i a_{i-1} \ (\forall i>1)   \ \rangle.$$
 Also, the subgroup $H :=  \langle a_1t, a_2t, \ldots \rangle$ of $G$ is $\displaystyle{\lim_{\longrightarrow} H_i}$ and $H_k = G_k \cap H$.  
 
 Our convention is that $[a,b]= a^{-1}b^{-1}ab$.   By re--expressing the original relations as $[a_1,t]=1$ and $a_{i-1} = [a_i,t]$ for $i>1$  and then eliminating $a_1, \ldots, a_{k-1}$ and defining $a:= a_k$,
one can  present $G_k$ with one relation, a nested  commutator, known as an \emph{Engel} relation:
$$G_k \ \cong \  \langle \   a,  t \  \mid \  [a, \underbrace{t, \ldots, t}_k ]  =1 \ \rangle.$$  
That is, the relation is $v_k=1$ where $v_k$ is the word  defined recursively by $v_0 = a$ and $v_{k+1} = [v_k,t]$ for $k \geq 0$.

Recursively define   a family of words by $u_0= a$ and $u_{k+1} = {u_k}^{-1}s{u_k}$ for $k \geq 0$.
  By inducting on $k$, one can verify that after substituting $t^{\pm 1}$ for every $s^{\mp 1}$ in $u_k$, the words $t^{-(k-1)} u_k t^k$ and $v_k$ become freely equal for all $k \geq 1$.
 So the relation $v_k=1$ can be replaced by $u_k =s$ to give  an alternative one--relator presentation for $G_k$:  $$G_k  \ \cong \     \left\langle \left. \rule{0mm}{5mm}    \   a,  s \  \right| \   \underbrace{\mbox{$s$\raisebox{4pt}{\reflectbox{$\ddots$}} \!\raisebox{14pt}{$s$}}}_k\!\raisebox{19pt}{$a$}  =  s \  \right\rangle.$$

This presentation can be re--expressed via $\alpha_i := u_{k-i}$ for $1 \leq i \leq k$  as
$$G_k  \ \cong \   \left\langle \     \alpha_1, \ldots, \alpha_k, s \  \left| \      {\alpha_1}^{-1}s{\alpha_1} =s, \  {\alpha_i}^{-1} s {\alpha_i} =\alpha_{i-1} \   ( i >1)  \ \right. \right\rangle.$$ By checking the link condition (see, for example, \cite[II.5.24]{BrH}) one finds that the Cayley 2-complex of this presentation (that is, the universal cover of the associated  presentation 2--complex), metrized so that each 2-cell is a Euclidean square, is  $\textup{CAT}(0)$.   Gersten \& Short \cite{GerstenShort2} proved that all such groups are automatic, and later Niblo \& Reeves~\cite{NR} proved that a more general class of groups, those acting geometrically on $\textup{CAT}(0)$ cube complexes, are biautomatic.

The groups $G_k$ enjoy  the property of \emph{rapid decay} as a consequence of \cite[Corollary~2.1.10]{Jolissaint}.  

We remark that a corollary of our recursive upper bound on $\Dist^{G_k}_{H_k}$ is that the membership problem for $H_k$ in $G_k$ is decidable. 

The family $G_k$ have received attention elsewhere.   In \cite{Gersten12} Gersten showed the group $G_2$   to be  $\textup{CAT}(0)$ with quadratic divergence  function.  He gave the free--by--cyclic, the one--relator, and the $\textup{CAT}(0)$ presentations of $G_2$ we have described.  In  \cite{Macura2} Macura shows $G_3$ to be $\textup{CAT}(0)$ and proves that an associated  $\textup{CAT}(0)$ complex  has a cubic divergence function.  Results in \cite{Macura2} imply that the universal cover of the divergence function of the mapping torus associated to the free--by--cyclic presentation of $G_k$ is polynomial  of degree $k-1$ (up to $\simeq$) and in  \cite{Macura4} Macura proves the same result for $\textup{CAT}(0)$ spaces associated to each $G_k$.   Macura also mentions $G_2$ and  $G_3$  in \cite{Macura1} as examples in the context of Kolchin maps and quadratic isoperimetric functions, and she and Cashen use $G_k$ as examples in \cite{CM} when studying novel   quasi--isometry invariants they call \emph{line patterns}.  It is stated in \cite{BBMS} (Example 4) that $G_3$ is biautomatic.
Bridson uses $G_k$ in \cite{Bridson2} as a starting point to construct  free--by-free groups with Dehn functions that are polynomial of degree $k+1$ and he shows them to be subgroups of $\textup{Out}(F_n)$ for suitable $n$.   Additionally he shows his example are asynchronously automatic via  normal forms which have length $\simeq n^k$,  but   by no shorter normal form.     En route he shows (Section 4.1(3)) that free--by--cyclic $F_k \rtimes \Z$  groups, such as $G_k$, embed in $\textup{Aut}(F_k)$.

Examples of yet more extreme distortion are known, even for subgroups of hyperbolic groups.  Arzhantseva \& Osin~\cite[\S3.4]{AO} and Pittet~\cite{PittetThesis} explain an argument  attributed to Sela in \cite[\S3, $3.K_3''$]{Gromov}: the Rips construction, applied to a finitely presentable group with unsolvable word problem yields a hyperbolic (indeed, $C'(1/6)$ small--cancellation) group $G$ with a finitely generated subgroup $N$ such that $\Dist^G_N$ is not bounded above by any recursive function.  The reason is that when $N$ is a finitely generated normal subgroup of a finitely presented group $G$,  there is an upper bound for the Dehn function of  $G/N$ in terms of the Dehn function of $G$ and the distortion of $N$ in $G$  --- see~\cite[Corollary~8.2]{Farb}, \cite{PittetThesis}.   Ol'shanskii \& Sapir in  \cite[Theorem~2]{OS} provide  another source of extreme examples --- using  Mikhailova's construction as their starting point, they show that the set of distortion functions of finitely generated subgroups of $F_2 \times F_2$ coincides (up to $\simeq$) with the set of Dehn functions of finitely presented groups.   As for finitely presented subgroups,   Baumslag, Bridson, Miller and Short \cite{BBMS2} explain how to construct groups $\Gamma$ that are both  $\textup{CAT}(0)$ and hyperbolic  and yet such that $\Gamma \times \Gamma$ has a finitely presented subgroup whose distortion is not bounded above by any recursive function. 

We are not aware of any systematic study of subgroup distortion in one--relator groups.  It seems natural to ask whether our examples are best--possible --- that is, whether there is a one--relator group with a finite--rank free subgroup of distortion $\succeq A_k$ for every $k$. 

\subsection{Extreme Dehn functions}

The Dehn function $\Area(n)$ of a finitely presented group $\langle A \mid R \rangle$ is related to the group's word problem in that $\Area(n)$ is the minimal $N$ such that given any word $w$ of length at most $n$ that represents the identity, $w$ freely equals some product $\prod_{i=1}^{N'} {u_i}^{-1} r_i u_i$ of $N' \leq N$ conjugates of relators $r_i \in R^{\pm 1}$, or, equivalently,  one can reduce $w$ to the empty word  by applying defining relations at most $N$ times and removing or inserting inverse pairs of letters.  At the same time, the Dehn function is a natural geometric invariant (in fact,  a quasi--isometry invariant up to $\simeq$) of a group: $\Area(n)$ is the minimal $N$ such that any edge--loop of length at most $n$ in the Cayley 2-complex of $\langle A \mid R \rangle$ can be spanned by a combinatorial filling disc (a van Kampen diagram) with area (that is, number of 2-cells) at most $N$.   This geometric perspective is related to the classical notion of an isoperimetric function in Riemannian geometry in that if  $\langle A \mid R \rangle$ is the fundamental group of a closed Riemannian manifold $M$, then its Dehn function is $\simeq$--equivalent to the minimal isoperimetric function of the universal cover of $M$.

Theorem~\ref{main} leads to strikingly simple examples of finitely presented groups with huge Dehn functions, namely the HNN--extensions of $G_k$ with stable letter commuting with all elements of the subgroup $H_k$.
\begin{thm} \label{Dehn function}
For $k \geq 2$, the Dehn function of the group
 $$\Gamma_k  \  := \ \langle  \ a_1, \ldots, a_k, t, p \  | \  t^{-1} {a_1} t = a_1, \ t^{-1} {a_i} t = a_i a_{i-1} \  (i > 1), \  [p, a_i t] =1 \  (i > 0) \  \rangle.$$
is $\simeq$--equivalent to $A_k$.   
\end{thm}

So, together with $\Gamma_1$, which has Dehn function $\simeq$--equivalent to  $n \mapsto n^2$ (see Proposition~\ref{quadratic Dehn function}),  these groups have Dehn functions that are representative  of each graduation of the Grzegorczyk hierarchy of primitive recursive functions.  Details of the proof are in Section~\ref{Big Dehn}.

These are not the only such examples  (but we believe they are the first that are explicit and elementary):
Cohen, Madlener and Otto \cite{CohenWisdom, CMO, MO} embedded algorithms  (\emph{modular Turing machines}, in fact) with running times like $n \mapsto A_k(n)$ in groups so that the running of the algorithm is displayed in van~Kampen diagrams so as to make the Dehn function reflect  the time--complexity of the algorithms.
 They state that their techniques produce yet more extreme examples as they also apply to an algorithm with running time like $n \mapsto A_n(n)$, and so yield a group with Dehn function that is recursive but not primitive recursive.
More extreme still, any finitely presentable group with undecidable word problem is not bounded above by any recursive function.

Elementary examples of groups with large Dehn function are described by Gromov in \cite[\S4]{Gromov}, but their behaviour is not so extreme.  There is the family
$$\langle \  x_0, \ldots, x_k  \ \mid \  {x_{i+1}}^{-1}{x_i}{x_{i+1}} = {x_i}^2 \ (i<k)  \  \rangle,$$
which has Dehn function $\simeq$--equivalent to $n \mapsto {\exp_2}^{(k)}(n)$.    [We write $\exp_2(n)$ to denote $2^n$.]
 And  Baumslag's group  \cite{Baumslag2}
 \begin{equation} \label{Baumslag's group}
 \langle \  a,b  \ \mid \   (b^{-1}a^{-1} b)  \, a \, (b^{-1}a b) = a^{2} \  \rangle,
 \end{equation}
which contains $\langle \  x_0, \ldots, x_k  \ \mid \  {x_{i+1}}^{-1}{x_i}{x_{i+1}} = {x_i}^2 \ (i \geq 0)  \  \rangle$ as a normal subgroup,  was shown by Platonov~\cite{Platonov} to have Dehn function $\simeq$--equivalent to  $n \mapsto {\exp}^{(\lfloor\log_2 n \rfloor ) }(1)$.  (Prior partial results towards  this direction are in \cite{Bernasconi, Gersten6, Gersten}.)

\subsection{The organisation of the article.}

We believe the most compelling assertion of Theorem~\ref{main} to be the existence of groups $H_k$ and $G_k$ with $H_k$ free of rank $k$, $G_k$ enjoying the bulleted list of properties, and $\Dist_{H_k}^{G_k}$ bounded below by $A_k$.  In particular, this shows that there is no uniform upper bound on the level in the Grzegorczyk hierarchy at which the functions $\Dist_{H_k}^{G_k}$ appear.  The reader who is primarily interested in these components of Theorem~\ref{main} need only read up to the end of Section~\ref{DistortionLower}.  The contents of this half of the article are as follows.  In Section~\ref{Ackermann section} we derive a collection of elementary properties of the Ackermann functions that will be used elsewhere in the paper.  Section~\ref{hydra functions versus Ackermann's functions} contains a proof of Proposition~\ref{prop2} comparing the hydra functions to Ackermann's functions.  In Section~\ref{freedom} we prove  that the subgroups $H_k$ are free.  And in Section~\ref{DistortionLower} we prove that each function $\Dist_{H_k}^{G_k}$ is bounded below by $\H_k$ --- combining this result with Proposition~\ref{prop2} gives the lower bound $A_k$.

Our proof that each function $\Dist_{H_k}^{G_k}$ lies in the same $\simeq$-equivalence class of functions as $A_k$ --- \ie that $A_k$ is an upper bound for $\Dist_{H_k}^{G_k}$ --- is considerably more involved than that of the lower bound and occupies most of the second half of the article: Sections~\ref{Upper bound preliminaries1}, \ref{Upper bound preliminaries2} and \ref{DistortionUpper}.  In deriving the upper bound, a key notion will be that of passing a power of $t$ through a word $w$ on the letters $a_i$.  We explain this idea in Section~\ref{Upper bound preliminaries1}, where we also identify some recursive structure that will be crucial in facilitating an inductive analysis.  In Section~\ref{Upper bound preliminaries2} we focus in detail on the situation where $w$ is of the form $\theta^n({a_k}^{\pm1})$ and derive preliminary result that will feed into the main proof, presented in Section~\ref{DistortionUpper}, that $\Dist_{H_k}^{G_k} \preceq A_k$.

Finally, in Section~\ref{Big Dehn}, we prove Theorem~\ref{Dehn function}, which gives the Dehn functions of the groups $\Gamma_k$.

We illustrate some of our arguments using van~Kampen diagrams, particularly observing their \emph{corridors} (also known as \emph{bands}).  For an introduction see, for example, I.8A.4 and the proof of Proposition~6.16 in III.$\Gamma$ of \cite{BrH}.    

We denote the length of a word $w$ by $\ell(w)$.  We write $w=w(a_{1}, \ldots, a_k)$ when $w$ is a word on $ {a_{1}}^{\pm 1}, \ldots, {a_k}^{\pm 1}$.

\subsection{Acknowledgements}  We are grateful to Martin~Bridson for a number of conversations on this work, to Indira~Chatterji for pointing out that the groups $G_k$ enjoy the rapid decay property, to Volker~Diekert for a discussion of Ackermann's functions, to Arye~Juhasz for background on one--relator groups, and to John~McCammond for help with some computer explorations of Hercules' battle with the hydra.

\section{Ackermann's functions} \label{Ackermann section}

Throughout this article we will frequently compare functions to Ackermann's functions and will find the following relationships useful.

\begin{lemma} \label{Properties of A_k}
   For integers $k,l, m, n$, the following relations hold within the given domains: \begin{align}
    A_{k}(A_{k+1}(n)) \ &= \ A_{k+1}(n + 1), &&k,n \geq 0, \label{A_k4}\\
    A_k(1) \ &= \ 2, &&k \geq 1, \label{A_k of 1} \\
    A_k(2) \ &= \ 4, &&k \geq 0, \label{A_k of 2} \\
    A_k(n) \ &\leq \ A_{k+1}(n), &&k \geq 1; n \geq 0,  \label{A_k7} \\
    A_k(n) \ &< \ A_k(n+1),  &&k,n \geq 0, \label{A_k5} \\
    n \ &\leq \ A_k(n), &&k,n \geq 0, \label{A_k2} \\
    \intertext{\textup{(}with equality holding in \eqref{A_k2} if and only if $(k,n) = (1,0)$\textup{)}}
     mA_k(n) \ &\leq \ A_k(nm), &&k, n \geq 1; m \geq 0, \label{A_k1} \\
     m {A_k}^{(l)}(n) \ &\leq \ {A_k}^{(l+m)}(n),  &&k \geq 1; l, m, n \geq 0,  \label{A_k6} \\
    A_k(n) + A_k(m) \ &\leq \ A_k(n + m), &&k,n,m \geq 1, \label{A_k3} \\
    A_k(n) + m \ &\leq \ A_k(n+m), &&k,n,m \geq 0, \label{A_k9} \\
    (A_k(n))^m \ &\leq \ A_k(nm), &&k \geq2; n, m \geq 0.  \label{A_k8}
  \end{align}
\end{lemma} 

\begin{proof}
  Equation \eqref{A_k4} follows immediately from the definition of the Ackermann functions.  Equations \eqref{A_k of 1} and \eqref{A_k of 2} follow from \eqref{A_k4} by an easy induction on $k$.

  Before proving \eqref{A_k7}, \eqref{A_k5} and \eqref{A_k2}, we first prove that non--strict versions of these inequalities hold.  The proof is by induction on $k$ and $n$.  It is easy to check that \eqref{A_k7} holds if $k=1$ or if $n=0$ and that \eqref{A_k5} and \eqref{A_k2} hold if $k=0$, if $k=1$ or if $n=0$.  Now let $k' >1$ and $n' > 0$ and suppose, as an inductive hypothesis, that \eqref{A_k7}, \eqref{A_k5} and \eqref{A_k2} hold (not necessarily strictly) if $k <k'$ or if $k=k'$ and $n<n'$.  We prove that the inequalities hold if $k=k'$ and $n=n'$.  For \eqref{A_k7}, we calculate that $A_{k'}(n') = A_{k'-1}(A_{k'}(n'-1)) \leq A_{k'-1}(A_{k'+1}(n'-1)) \leq A_{k'}(A_{k'+1}(n'-1)) = A_{k'+1}(n')$, where we have applied \eqref{A_k4} and the inductive hypothesis versions of \eqref{A_k7} and \eqref{A_k5}.  For \eqref{A_k5}, we calculate that $A_{k'}(n') \leq A_{k'-1}(A_{k'}(n')) = A_{k'}(n'+1)$, where we have used \eqref{A_k4} and the inductive hypothesis version of \eqref{A_k2}.  For \eqref{A_k2}, we calculate that $n' \leq 2n' = A_1(n') \leq A_{k'}(n')$, where we have used the inductive hypothesis version of \eqref{A_k7}.  This completes the proof that \eqref{A_k7}, \eqref{A_k5} and \eqref{A_k2} hold in non-strict form.  Now observe that equality in \eqref{A_k2} at $(k,n) = (k', n')$ requires $n' = 2n'$, whence $n'=0$.  Since $A_k(0)=1$ for all $k \geq 2$, equality in \eqref{A_k2} holds if and only if $(k,n) = (1,0)$.  It follows that equality in \eqref{A_k5} at $(k,n) = (k', n')$ would require that $A_{k'}(n') =0$ and $k'-1 = 1$, whence $A_2(n') = 0$.  But $A_2(n) = 2^n >0$ for all $n$ and so the inequality \eqref{A_k5} is strict.

  We now prove inequality \eqref{A_k1}.  This clearly holds if $m=0$, so suppose that $m \geq 1$.  The proof is by induction on $k$ and $n$.  It is clear that \eqref{A_k1} holds if $k=1$.  The inequality also holds if $n=1$ since, applying \eqref{A_k of 1} and \eqref{A_k7}, we calculate that $mA_k(1) = 2m = A_1(m) \leq A_k(m)$.  Now let $k', n' > 1$ and suppose, as an inductive hypothesis, that \eqref{A_k1} holds if $k<k'$ or if $k=k'$ and $n <n'$.  We calculate that $mA_{k'}(n') = m A_{k'-1}(A_{k'}(n'-1)) \leq A_{k'-1}(mA_{k'}(n'-1)) \leq A_{k'-1}(A_{k'}(mn'-m)) \leq A_{k'-1}(A_{k'}(mn'-1)) = A_{k'}(mn')$, where we have used \eqref{A_k4} and \eqref{A_k5}.  Thus the inequality holds if $(k,n) = (k',n')$, completing the proof of \eqref{A_k1}.

  For inequality \eqref{A_k6} observe that, by \eqref{A_k2}, $m {A_k}^{(l)}(n) \leq A_{k+1}(m) {A_k}^{(l)}(n) = {A_k}^{(m)}(1) {A_k}^{(l)}(n)$.  It also follows from \eqref{A_k2} that ${A_k}^{(i)}(1) \geq 1$ for all $i \geq 0$.  We can thus apply \eqref{A_k1}, together with \eqref{A_k5}, to show that ${A_k}^{(m)}(1) {A_k}^{(l)}(n) \leq {A_k}^{(m)}({A_k}^{(l)}(n)) = {A_k}^{(l+m)}(n)$.

  We prove \eqref{A_k3} by induction on $k$.  It is clear that the inequality holds if $k=1$, so suppose that $k >1$ and that the result is true for smaller values of $k$.  Without loss of generality suppose that $n \leq m$.  It follows from \eqref{A_k2} that ${A_{k-1}}^{(i)} \geq 1$ for all $i \geq 0$, and so we can apply the induction hypothesis to calculate that $A_k(n) + A_k(m) = {A_{k-1}}^{(n)}(1) + {A_{k-1}}^{(m)}(1) \leq {A_{k-1}}^{(n)}(1 + {A_{k-1}}^{(m-n)}(1)) = {A_{k-1}}^{(n)}(1 + A_k(m-n))$.  Applying \eqref{A_k5} gives that this quantity is at most ${A_{k-1}}^{(n)}(A_k(m-n+1)) = A_k(m+1) \leq A_k(m+n)$.

  We now prove inequality~\eqref{A_k9}.  This clearly holds if $k=0$, $k=1$ or $m=0$.  If $k \geq 2$ and $n=0$, then $A_k(n) + m = m+1 \leq A_k(m) = A_k(n+m)$ by \eqref{A_k2}.  It remains to prove \eqref{A_k9} if $k,n,m \geq 1$. But in this case $A_k(n) + m \leq A_k(n) + A_k(m) \leq A_k(n+m)$ by \eqref{A_k2} and \eqref{A_k3}.

  Finally, we prove \eqref{A_k8} by induction on $k$.  It is clear that the inequality holds if $k=2$, so suppose that $k \geq 3$ and that the result holds for smaller values of $k$.  It is also clear that the inequality holds if $n = 0$ or if $m = 0$; suppose that $n,m \geq 1$.   Applying the induction hypothesis, together with \eqref{A_k4}, we calculate that $A_k(n)^m = A_{k-1}(A_k(n-1))^m \leq A_{k-1}(mA_k(n-1))$.  Applying \eqref{A_k4}, \eqref{A_k5} and \eqref{A_k1}, we see that this quantity is at most $A_{k-1}(A_k(nm - m)) \leq A_{k-1}(A_k(nm - 1)) = A_k(nm)$.
\end{proof}

\section{Comparing the hydra functions to  Ackermann's functions}
  \label{hydra functions versus Ackermann's functions}

In this section we prove Proposition~\ref{prop2} comparing Ackermann's functions to the hydra functions.  The proof will proceed via a third family of functions $\phi_k$.  In this section $\phi_k(n)$ will will be defined for $n \geq 0$; subsequently we will give a more general definition with an expanded domain.

For integers $k \geq 1$ and $n \geq 0$, define $\phi_k(n) := \H(\theta^n(a_k))$.  The functions $\H_k$ satisfy the recursion relation \begin{equation} \H_k(n+1) = \H_k(n) + \phi_k(\H_k(n)) \label{eq4} \end{equation} since after $\H_k(n)$ strikes the word ${a_k}^{n+1}$ has become $\theta^{\H_k(n)}(a_k)$.  We will need the following elementary properties of the functions $\phi_k$.

\begin{lemma} \label{lem1}
  For integers $k \geq 1$ and $n \geq 0$,  \begin{align}
    \phi_k(0) \ &= \ 1, \label{phi of 0} \\
    \phi_2(n) \ &= \ n+1, \label{phi2} \\
    \phi_k(n) \ &\geq \ 1, \label{phi geq 1} \\
    \phi_{k+1}(n+1) \ &= \ \phi_{k+1}(n) + \phi_k(\phi_{k+1}(n) + n). \label{phi recursion relation}
    \intertext{For integers $k \geq 2$ and $n \geq 0$,}
    \phi_k(n) \ &< \ \phi_k(n+1), \label{phi increasing} \\
    \phi_k(n) \ &\geq \ n. \label{phi geq n}
  \end{align}
\end{lemma}

\begin{proof}
  Assertions \eqref{phi of 0}, \eqref{phi2} and \eqref{phi geq 1} are straightforward.  For \eqref{phi recursion relation}, note that, by induction on $n$, $\theta^{n+1}(a_{k+1})  = a_{k+1}  a_{k}  \theta(a_{k}) \ldots \theta^{n}(a_{k})$ and hence $\theta^{n+1}(a_{k+1}) = \theta^{n}(a_{k+1}) \theta^{n}(a_{k})$.  Thus, after $\phi_{k+1}(n)$ strikes, $\theta^{n+1}(a_{k+1})$ has become $\theta^{\phi_{k+1}(n)}(\theta^{n}(a_{k})) = \theta^{\phi_{k+1}(n) + n}(a_{k})$.  Inequality \eqref{phi increasing} follows immediately from \eqref{phi geq 1} and \eqref{phi recursion relation} and inequality~\eqref{phi geq n} follows from \eqref{phi geq 1} and \eqref{phi increasing}.
\end{proof}

It is easy to check that $\phi_1 \simeq A_0$ and $\phi_2 \simeq A_1$.  As such, the next result is sufficient to establish that $\phi_k \simeq A_{k-1}$ for $k \geq 1$.

\begin{lemma} \label{phi_k vs Ackermann}
  \mbox{}
  \begin{enumerate}
    \item For integers $k \geq 3$ and $n \geq 0$, $\phi_k(n) \geq A_{k-1}(n)$.

    \item For integers $k \geq 2$ and $n \geq 0$, $\phi_k(n) \leq A_{k-1}(n+k) - n - k$.
  \end{enumerate}

\end{lemma}

\begin{proof}
  We prove (i) by simultaneous induction on $k$ and $n$.  It is immediate from \eqref{phi of 0} that the inequality holds if $n = 0$.  Solving the recursion relation \eqref{phi recursion relation} with the initial condition given by \eqref{phi of 0}, one checks that $\phi_3(n) = 3\cdot 2^n - n - 2$.  Since $A_2(n) = 2^n$, it is easy to check that (i) holds if $k=3$. Now let $k' >3$ and $n' > 0$ and suppose, as an inductive hypothesis, that the result is true if $k < k'$ or if $k = k'$ and $n < n'$.  Applying \eqref{A_k4}, \eqref{phi geq 1} and \eqref{phi increasing}, we calculate that $\phi_{k'}(n') = \phi_{k'}(n'-1) + \phi_{k'-1}(\phi_{k'}(n'-1) + n' - 1) \geq \phi_{k'-1}(\phi_{k'}(n'-1)) \geq \phi_{k'-1}(A_{k'-1}(n'-1)) \geq A_{k'-2}(A_{k'-1}(n'-1)) = A_{k'-1}(n')$.  Thus the result holds at $(k,n) = (k', n')$, completing the proof of (i).

  We now make the following claim: for all $k \geq 2$, $n \geq 0$ and $c \geq k$, \begin{equation}
    \phi_k(n) \  \leq \  A_{k-1}(n+c) - n + k - 2c. \label{claim}
  \end{equation}  Assertion (ii) will follow by setting $c = k$.  The proof of this inequality is by simultaneous induction on $k$ and $n$.  Since $A_1(n) = 2n$ and, by \eqref{phi2}, $\phi_2(n) = n+1$, it is straightforward to check that \eqref{claim} holds if $k=2$.  The inequality also holds for $n = 0$ since, by \eqref{A_k7} and \eqref{phi of 0}, $\phi_k(0) = 1 \leq k = A_1(c) + k -2c \leq A_{k-1}(c) + k - 2c$.  Now let $c \geq k' > 2$ and $n' > 0$ and suppose, as an induction hypothesis, that \eqref{claim} holds if $k < k'$ or if $k = k'$ and $n < n'$.  We calculate that \begin{align*}\phi_{k'}(n') \ &=  \ \phi_{k'}(n'-1) + \phi_{k'-1}(\phi_{k'}(n'-1) + n'-1) &&\text{by \eqref{phi recursion relation}} \\
  &\leq \ \phi_{k'}(n'-1) + A_{k'-2}(\phi_{k'}(n'-1) + n' + c - 1) \\
  & \mspace{50mu} - \phi_{k'}(n'-1) - n' + k' - 2c \\
  &= \ A_{k'-2}(\phi_{k'}(n'-1) + n' + c - 1) - n' + k' -2c  \\
  &\leq \ A_{k'-2}(A_{k'-1}(n'+c-1) + k' - c) - n' + k' - 2c && \text{by \eqref{A_k5}} \\
  &\leq \ A_{k'-2}(A_{k'-1}(n'+c-1)) - n' + k' - 2c && \text{by \eqref{A_k5}} \\
  &= \ A_{k'-1}(n' + c) - n' + k' -2c. && \text{by \eqref{A_k4}} \end{align*}  Thus the inequality holds if $(k,n) = (k', n')$, completing the proof of \eqref{claim}.
\end{proof}

Since $A_1(n) = 2n$, $\H_1(n) = n$, $A_2(n) = 2^n$ and $H_2(n) = 2^n - 1$, the next result is sufficient to establish Proposition~\ref{prop2}.

\begin{prop} \label{prop3}
  \mbox{}
  \begin{enumerate}
    \item For integers $k \geq 3$ and $n \geq 2$, $\H_k(n) \geq A_k(n)$.

    \item For integers $k \geq 1$ and $n \geq 0$, $\H_k(n) \leq A_k(n+k)$.
  \end{enumerate}
\end{prop}

\begin{proof}[Proof of Proposition~\ref{prop3}]
   We prove (i) by induction on $n$.  The inequality certainly holds for $n=2$ since, by \eqref{A_k of 2}, $\H_k(2) = \H(a_k a_{k-1} a_{k-1} a_{k-2}) \geq 4 = A_k(2)$.  Now let $n' > 2$ and suppose that (i) holds for $n < n'$.  Applying \eqref{A_k4}, \eqref{eq4} and \eqref{phi increasing}, together with Lemma~\ref{phi_k vs Ackermann} (i), we calculate that $\H_k(n') = \H_k(n'-1) + \phi_k(\H_k(n'-1)) \geq \phi_k(\H_k(n'-1)) \geq \phi_k(A_k(n'-1)) \geq A_{k-1}(A_k(n'-1)) = A_k(n')$.  Thus the inequality holds for $n=n'$, completing the proof of (i).

  For (ii), we prove the stronger claim that, for all $k \geq 1$, $n \geq 0$, \begin{equation}
    \H_k(n) \ \leq \ A_k(n+k) - k. \label{eq3}
  \end{equation}  The proof is by simultaneous induction on $k$ and $n$.  Since $A_1(n) = 2n$ and $\H_1(n) =n$, it is straightforward to check that \eqref{eq3} holds if $k=1$.  The inequality holds if $n=0$ since, by \eqref{A_k7}, $\H_k(0) = 0 \leq k = A_1(k) - k \leq A_k(k) - k$.  Now let $k' >1$ and $n'>0$ and suppose, as an inductive hypothesis, that \eqref{eq3} holds if $k < k'$ or if $k = k'$ and $n < n'$.  We calculate that \begin{align*}
    \H_{k'}(n') & \ = \  \H_{k'}(n'-1) + \phi_{k'}(\H_{k'}(n'-1)) &&\text{by \eqref{eq4}} \\
    & \ \leq \  \H_{k'}(n'-1) + A_{k'-1}(\H_{k'}(n'-1) + k') \\
     &\mspace{50mu} - \H_{k'}(n'-1) - k' &&\text{by Lemma~\ref{phi_k vs Ackermann} (ii)}\\
    & \ =  \ A_{k'-1}(\H_{k'}(n'-1) + k') - k' \\
    & \ \leq \  A_{k'-1}(A_{k'}(n'+k'-1)) - k' &&\text{by \eqref{A_k5}} \\
    & \ = \ A_{k'}(n' + k') - k' &&\text{by \eqref{A_k4}.}
  \end{align*} Thus the inequality holds if $(k,n) = (k',n')$, completing the proof of \eqref{eq3}.
\end{proof}

\section{\texorpdfstring{Freeness of the subgroups $H$ and $H_k$}{Freeness of the subgroups H and Hk}}  \label{freedom}

In this section we prove:

\begin{prop} \label{freeness of H_k}  \label{freeness of H}
The subgroup $H_k$ of $G_k$ is free with free basis $a_1t, \ldots, a_kt$, and
the subgroup $H$ of $G$ is free with free basis $a_1t, a_2t, \ldots$.
\end{prop}

To facilitate an induction argument, we will prove the following more elaborate proposition.     Proposition~\ref{freeness of H_k} will follow because if $w = w(a_1t, \ldots, a_kt)$ is freely reduced and represents $1$ in $G_k$ (or, equivalently, in $G$), then $w = \varepsilon$ by  conclusion (i), and so  $a_1t,  \ldots, a_kt$ are each not the identity and satisfy no non-trivial relations.

 \begin{prop} \label{prop5}
  Let $u = u(a_1 t, \ldots, a_k t)$ be a freely reduced word with free--by--cyclic normal form $v t^r$ --- that is,  $u = v t^r$ in $G_k$,   $v= v(a_1, \ldots, a_k)$ is reduced, and $r \in \Z$.  \begin{enumerate}
    \item If $v = \varepsilon$, then $u = \varepsilon$.
    \item If $v = \theta({a_{k+1}}^{-1}) \theta^{1-r}(a_{k+1})$ in $F(a_1, a_2, \ldots)$, then $u = \varepsilon$.
    \item If $v$ is positive, then $u$ is positive.
  \end{enumerate}
\end{prop}

We emphasise that we are considering $u$ as a word on the $a_it$ --- it is freely reduced if and only if it contains no subword $(a_it)^{\pm1} (a_it)^{\mp1}$.

\begin{proof}[Proof of Proposition~\ref{prop5}]
  We first show that for all fixed $k \geq 1$, if (iii) holds, then so do (i) and (ii).

  For (i), note that if $u = t^r$ in $G$, then $u^{-1} = t^{-r}$.  Thus (iii) implies that both of the freely reduced words $u$ and $u^{-1}$ are positive.  Hence $u = \varepsilon$.

  For (ii), we will separately consider the cases $r=0$, $r<0$, and $r>0$.  If $r = 0$, then $u = 1$ in $G$ and hence $u = \varepsilon$ by (i).  If $r < 0$, then $1-r \geq 1$ and so $\theta^{1-r}(a_{k+1}) = a_{k+1} a_k w$ in $F(a_1, a_2, \ldots)$ for some positive word $w = w(a_1, \ldots, a_k)$.  It follows that $v$ is positive and therefore (iii) implies that $u$ is positive.  Thus $r \geq 0$, giving a contradiction.  If $r > 0$, one calculates that $u^{-1} = t^{-r} \theta^{1-r}({a_{k+1}}^{-1}) \theta(a_{k+1}) = \theta({a_{k+1}}^{-1}) \theta^{1+r}(a_{k+1}) t^{-r}$  in $F(a_1, a_2, \ldots)$.  Since $1 + r \geq 1$, the reduced form of $\theta({a_{k+1}}^{-1}) \theta^{1+r}(a_{k+1})$ is positive, and so (iii) implies that $u^{-1}$ is positive.  Thus $-r \geq 0$, giving a contradiction.

 We now prove (iii) by induction on $k$.  Since $G_1$ is free abelian with basis $a_1, t$, it is easy to check that (iii) holds in the case $k=1$.  As an inductive hypothesis, assume that assertions (i), (ii) and (iii) all hold for smaller values of $k$. If $u$ contains no occurrence of an $(a_k t)^{\pm1}$, then we are done.  Otherwise, write $u = \sigma_0 (a_k t)^{\epsilon_1} \sigma_1 (a_k t)^{\epsilon_2} \ldots (a_k t)^{\epsilon_m} \sigma_m$, where each $\sigma_i = \sigma_i(a_1t, \ldots, a_{k-1}t)$ and each $\epsilon_i \in \{\pm1\}$.

Each $\sigma_i$ has free--by--cyclic normal form $\tau_i t^{s_i}$ for some $\tau_i = \tau_i(a_1, \ldots, a_{k-1})$ and some $s_i \in \Z$.  Direct calculation of the normal form of $u$ --- moving all the $t^{\pm 1}$ to the right--hand end and applying the automorphism $\theta^{\mp 1}$ whenever a $t^{\pm1}$ is moved past a letter $a_i$ --- gives that $v$ freely equals $$v' \ := \ \tau_0 \, \theta^{\lambda_1}({a_k}^{\epsilon_1}) \, \theta^{\mu_1}(\tau_1) \, \theta^{\lambda_2}({a_k}^{\epsilon_2}) \ldots \theta^{\lambda_m}({a_k}^{\epsilon_m}) \, \theta^{\mu_m}(\tau_m),$$ where \begin{align*}
    \lambda_i &= \begin{cases}
      -(s_0 + \ldots + s_{i-1} + \epsilon_1 + \ldots + \epsilon_{i-1}) \quad &\text{if $\epsilon_i = 1$,} \\
      -(s_0 + \ldots + s_{i-1} + \epsilon_1 + \ldots + \epsilon_i) \quad &\text{if $\epsilon_i = -1$,}
    \end{cases} \\[5pt]
    \mu_i &= -(s_0 + \ldots + s_{i-1} + \epsilon_1 + \ldots + \epsilon_i).
  \end{align*}

We claim that $\epsilon_i = 1$ for all $i$.
      For a contradiction, suppose otherwise.  Observe that, for each $s \in \Z$, there are words $w_s=w_s(a_1, \ldots, a_{k-1})$ and $w'_s=w'_s(a_1, \ldots, a_{k-1})$ such that $\theta^s(a_k) = a_k w_s$ and $\theta^s({a_k}^{-1}) = w_s' {a_k}^{-1}$.  Since $v$ is positive, there must be a subword ${a_k}^{\pm 1} \chi \ {a_k}^{\mp 1}$ in $v'$ which freely equals the empty word and in which $\chi = \chi(a_1, \ldots, a_{k-1})$.  The way this subword must arise is that for some $i$, either 
        \begin{enumerate}
      \item[(a)] $\epsilon_i = -1$, $\epsilon_{i+1} = 1$ and $\theta^{\mu_i}(\tau_i) =1$, or
      \item[(b)] $\epsilon_i = 1$, $\epsilon_{i+1}=-1$ and $\theta^{\lambda_i}(a_k) \, \theta^{\mu_i}(\tau_i) \, \theta^{\lambda_{i+1}}({a_k}^{-1}) = 1$.
    \end{enumerate}  In the first case $\tau_i = 1$ and hence the induction hypothesis (assertion (i)) gives that $\sigma_i = \varepsilon$.  But this contradicts the supposition that $u$ is freely reduced.  In the second case, one calculates that $\lambda_i - \mu_i = 1$ and $\lambda_{i+1} - \mu_i = 1 - s_i$, and so $\tau_i = \theta({a_k}^{-1}) \theta^{1 - s_i}(a_k)$.  The induction hypothesis (assertion (ii)) implies that $\sigma_i = \varepsilon$, but again this contradicts the supposition that $u$ is freely reduced.

  To complete our proof of (iii), we will show that all the $\sigma_i$ are positive.  Since $v$ is positive and each $\epsilon_i = 1$, we have that $\tau_0$ is positive and each $\theta^{\lambda_i}(a_k) \, \theta^{\mu_i}(\tau_i)$ is positive.  The inductive hypothesis (assertion (iii)) immediately gives that $\sigma_0$ is positive.  Suppose we have shown that $\sigma_0, \ldots, \sigma_{j-1}$ are positive, for some $j$.  It follows that $s_0, \ldots, s_{j-1} \geq 0$, whence $\lambda_j \leq 0$.  Note that if $w=w(a_1, \ldots, a_k)$ is positive and $s \geq 0$, then $\theta^s(w)$ is positive.  Hence $a_k \theta^{\mu_j - \lambda_j}(\tau_j) = a_k \theta^{-1}(\tau_j)$ is positive.  Since $\theta^{-1}(\tau_j)$ is a word on ${a_1}^{\pm 1}, \ldots, {a_{k-1}}^{\pm 1}$, it follows that $\theta^{-1}(\tau_j)$ is positive, and hence that $\tau_j$ is positive.  Applying the induction hypothesis (assertion (iii)) gives that $\sigma_j$ is positive.
\end{proof}

\section{\texorpdfstring{A lower bound on the distortion of $H_k$ in $G_k$}{A lower bound on the distortion of Hk in Gk}}  \label{DistortionLower}

In the following lemma  we see the battle between Hercules and the hydra manifest in $G_k$.  

\begin{lemma}  \label{big triangle}
For all $k,n \geq 1$, there is a positive word $u_{k,n} = u_{k,n}(a_1t, \ldots, a_kt)$ of length $\H_k(n)$ that equals ${a_k}^n t^{\H_k(n)}$ in $G_k$.
 \end{lemma}

\begin{proof}
Consider the following calculation in which successive $t$ are moved to the front and paired off with the $a_i$.  [We illustrate the calculation in the case $k \geq 3$ and $n \geq 2$ --- for $k=2$, the letters $a_{k-2}$ would not appear and for $k=1$, neither would the $a_{k-1}$.]
\begin{align*}
{a_k}^n t^{\H_k(n)}  & = \  (a_k t)   \  t^{-1} {a_k}^{n-1} t  \  t^{\H_k(n)-1}  \\
& = \  (a_k t) \      (a_k a_{k-1})^{n-1}   \  t^{\H_k(n)-1}  \\
& = \  (a_k t) \     (a_k t)   \  t^{-1} a_{k-1} (a_k a_{k-1})^{n-2}  t  \  t^{\H_k(n)-2}   \\
 & = \  (a_k t)    \  (a_k t)  \    a_{k-1} a_{k-2} (a_k a_{k-1}  a_{k-1} a_{k-2})^{n-2}  \  t^{\H_k(n)-2} \\
&  \ \ \vdots
\end{align*}
A van~Kampen diagram displaying this calculation in the case $k=2$ and $n=4$ is shown in Figure~\ref{triangular diagram}.  

 \begin{figure}[ht]
\psfrag{a}{$a_2$}
\psfrag{c}{$a_1$}
\psfrag{t}{$t$}
\centerline{\epsfig{file=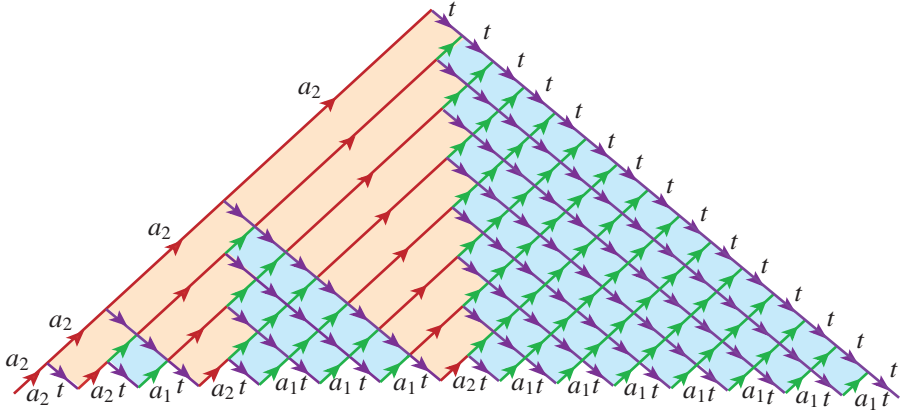}} \caption{A van~Kampen diagram showing that ${a_2}^4   t^{15} = u_{2,4}$ in $G_2$    where $u_{2,4} = a_2t \, a_2t \, a_1t \, a_2t \, (a_1t)^3 \, a_2t \, (a_1t)^7$.} \label{triangular diagram}
\end{figure}

One sees the  Hercules--versus--the--hydra battle 
$${a_k}^n \ \to \   (a_k a_{k-1})^{n-1}  \ \to \   a_{k-1} a_{k-2} (a_k a_{k-1} a_{k-1} a_{k-2})^{n-2} \  \to \  \cdots$$
being played out in this calculation. The pairing off of a $t$ with an $a_i$ corresponds to a decapitation, and the conjugation by $t$ that moves that $t$ into place from the right--hand end causes a hydra--regeneration for the intervening subword.   So by Proposition~\ref{Hercules always wins}, after $\H_k(n)$ steps we have  a positive word on $u_{k,n} = u_{k,n}(a_1t, \ldots, a_kt)$, and its length is $\H_k(n)$.
\end{proof}

Our next proposition establishes  that $\Dist_{H_k}^{G_k} \succeq \H_k$ for all $k \geq 2$.  The case $k=1$ is straightforward: $H_1 \cong \Z$ is undistorted in $G_1 \cong \Z^2$   and   $\H_1(n) =n$.   The calculation in the proof of the proposition is illustrated by a van~Kampen diagram in Figure~\ref{distortion diagram} in the case $k=2$ and $n=4$ --- the idea is that  a copy of the diagram from Figure~\ref{triangular diagram} fits together with its mirror image along intervening $a_1$-- and $a_2$--corridors to make a diagram demonstrating the equality of a freely reduced word of extreme length on $a_1t, \ldots, a_kt$ with a short word on $a_1, \ldots, a_k, t$.

\begin{prop}
For all $k \geq 2$ and $n \geq 1$, there is a reduced word of length $2 \mathcal{H}_k(n) +3$ on the free basis $a_1t, \ldots, a_kt$ for $H_k$ which, in $G_k$, equals a word of length $2n+4$ on $a_1, \ldots, a_k, t$.  
\end{prop}

\begin{proof} 
The relation $t^{-1} a_2 t  = a_2 a_1$ can be expressed as ${a_2}^{-1} t a_2 = t {a_1}^{-1}$, from which one obtains    ${a_2}^{-1}  t^{\H_k(n)} a_2 = (t {a_1}^{-1})^{\H_k(n)}$.  Combined with Lemma~\ref{big triangle} this gives  $$(t {a_1}^{-1})^{\H_k(n)} = {a_2}^{-1} {a_k}^{-n} u_{k,n} a_2.$$ So, as $a_1$ commutes with $t$, we deduce that $t a_1$ commutes with ${a_2}^{-1} {a_k}^{-n} u_{k,n} a_2$, and therefore $${a_k}^{n} a_2 \  t a_1 \  {a_2}^{-1} {a_k}^{-n}  \ = \ u_{k,n} a_2 \  t a_1 \  {a_2}^{-1} {u_{k,n}}^{-1}$$ in $G_k$.  The word on the left--hand side of this equation has length $2n +4$.  The word on the right--hand side freely equals $u_{k,n} \,  (a_2t) \,  (a_1 t) \, ({a_2}t)^{-1} \, {u_{k,n}}^{-1}$, which, viewed as a word on $a_1t, \ldots, a_kt$, is  freely reduced and has length $2 \mathcal{H}_k(n) + 3$, since $u_{k,n}$ is a positive word.     
\end{proof}

 \begin{figure}[ht]
\psfrag{a}{$a_2$}
\psfrag{b}{$a_1$}
\psfrag{A}{${a_2}^4$}
\psfrag{t}{$t$}
\psfrag{u}{$u_{2,4}$}
\centerline{\epsfig{file=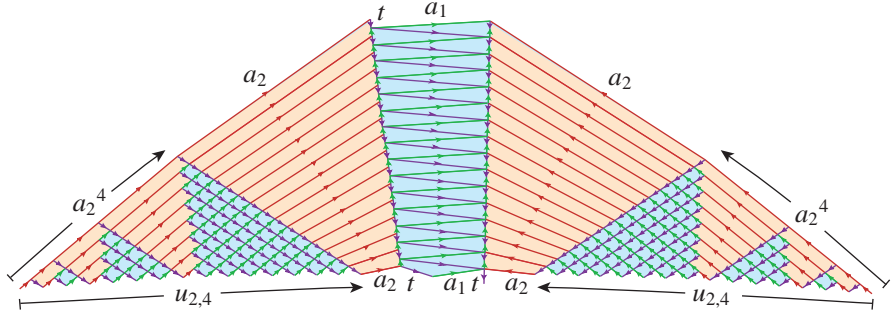}} \caption{A van~Kampen diagram demonstrating the equality \newline   ${a_2}^4 \, a_2 t a_1 {a_2}^{-1}   {a_2}^{-4} = u_{2,4} \, (a_2 t) \, (a_1t)  \, (a_2t)^{-1}  {u_{2,4}}^{-1}$ in $G_2$.} \label{distortion diagram}
\end{figure}

\section{Recursive structure of words} \label{Upper bound preliminaries1}

This section contains preliminaries that will feed into the proof, presented in Section~\ref{DistortionUpper}, that $\Dist^{G_k}_{H_k} \preceq A_k$.  In outline, we will bound  the distortion of $H_k$ in $G_k$  by first supposing $u = u(t,a_1,  \ldots, u_k)$ represents an element of $H_k$.  We will shuffle all the $t^{\pm 1}$ in $u$ to the start, with the effect of applying $\theta^{\pm 1}$ to each ${a_i}^{\pm 1}$ they pass.  After freely reducing, we will have a word $t^r w$ where $w = w(a_1, a_2, \ldots)$.   We will then look to carry the $t^r$ back through $w$ from left to right, converting all it passes to a word on $a_1t, a_2t, \ldots$ which will represent an element of $G_k$ since the indices of the $a_i$ will be no higher than those occurring in $w$.  Estimating the length of this word will give an upper bound on $\Dist^{G_k}_{H_k}$.

For convenience, we work with the group $G$ and its subgroup $H$ defined in Section~\ref{wild distortion}. 

When carrying the power of $t$ through $w$ we will face the problem of whether a word $t^r w$, where $w = w(a_1, a_2, \ldots)$, represents an element of a coset $Ht^s$ in $G$ for some $s \in \Z$.  We will see that the answer is not always affirmative --- these cosets do not cover $G$.
However, if  $t^r w = \sigma t^s$ for some $\sigma = \sigma(a_1t, a_2t, \ldots)$ and some $s \in \Z$, then  $\sigma$ is unique up to free--equivalence since $H$ is free  (Proposition~\ref{freeness of H}) and  $s$ is unique by our next lemma.  Indeed, this implies that $Ht^{s_1}$ and $Ht^{s_2}$ are equal precisely when $s_1 = s_2$.   

\begin{lemma} \label{t^k in H}
If $\ell \in \Z$ and $t^{\ell} \in H$, then ${\ell}=0$.
\end{lemma}

\begin{proof}
Were $t^{\ell} \in H$ for some integer ${\ell} \neq 0$, then $\Z^2 \cong  \langle a_1t, t^{\ell} \rangle$ would be a  subgroup of $H$ contra to the freeness of  $H$ established in Proposition~\ref{freeness of H}.
\end{proof}

Our next lemma will  be the crux of our method for establishing an upper bound on distortion.  It identifies recursive structure that will allow us to analyse the process of passing a power of $t$ through a word  $w = w(a_1, a_2, \ldots)$, so as to leave behind a word on $a_1 t, a_2t, \ldots$.  

 For a non--empty freely--reduced word $w= w(a_1, a_2, \ldots)$, define the \emph{rank of $w$} to be the highest $k$ such that ${a_k}^{\pm 1}$ occurs in $w$.  We define the empty word to have rank $0$.
For an integer $k \geq 1$, define a \emph{piece of  rank $k$} to be a freely--reduced word ${a_k}^{\epsilon_1}  \pi {a_k}^{-\epsilon_2}$ where $\pi = \pi(a_1, \ldots, a_{k-1})$ and $\epsilon_1, \epsilon_2 \in \set{0,1}$.  Notice that  a piece of rank $k$ will always also be a piece of rank $k+1$ and that the empty word is a piece of rank $k$ for every $k$.  

For a non--empty freely--reduced word $w$ of rank $k$, define the \emph{number of pieces in $w$} to be the least integer $m$ such that $w$ can be expressed as a concatenation  $w_1 \ldots w_m$ of subwords $w_i$ each of which is a piece of rank $k$.  (We say the empty word is composed of $0$ pieces.) 
 Observe that 
\begin{enumerate}
\item   each ${a_k}$ and ${a_k}^{-1}$ in $w$ is the first or last letter of some $w_i$, respectively; 
\item for $i=1, \ldots, m-1$, either the final letter of $w_i$ is ${a_k}^{-1}$ or the first of $w_{i+1}$ is $a_k$, but never both; and  
\item if ${a_k}^{-1} \chi  a_k$ is a subword of $w$ and $\chi = \chi(a_1, \ldots, a_{k-1})$, then  $\chi = w_i$ for some $i$. 
\end{enumerate}
In particular,  $w_1$, \ldots, $w_m$ are uniquely determined by the locations of the ${a_k}^{\pm 1}$ in $w$, and so we call the list of subwords $w_1, \ldots, w_m$ \emph{the partition of $w$ into pieces}.  
 
For example,  $w := {a_3}^{-1} a_1 a_2 a_3 {a_2}^{-1} a_3 {a_1}^{-1} {a_3}^{-1}$ has rank $3$ and its  partition into pieces is $w = w_1 w_2 w_3 w_4$ where $w_1 = {a_3}^{-1}$, $w_2 = a_1 a_2$, $w_3 = a_3 {a_2}^{-1}$, and $w_4 = a_3 {a_1}^{-1} {a_3}^{-1}$.

\begin{lemma} \label{Key Lemma}
Suppose $w = w(a_1, \ldots, a_k)$ is  a non--empty freely--reduced word of rank $k$ and $r$ and $s$ are integers such that  $t^r w \in H t^s$.  Let $w = w_1 \ldots w_n$ be the partition of $w$ into pieces.  Then there exist integers $r=r_0, r_1, \ldots, r_n =s$ such that $t^{r_i} w_{i+1} \in H t^{r_{i+1}}$ for each $i$.
\end{lemma}

\begin{proof}
As $t^rw \in H t^s$, there is some reduced word $v = v(a_1t,  \ldots, a_kt)$ such that $t^r w = v t^s$.   Form the analogue of a \emph{partition into pieces} for $v$ --- that is, express $v$ as  a concatenation $v_1 \ldots v_m$ of subwords $v_i$ each of the form  ${(a_k t)}^{\epsilon_1} \, \tau \, {(a_k t)}^{-\epsilon_2}$ where $\tau = \tau(a_1t, \ldots, a_{k-1}t)$ and $\epsilon_1, \epsilon_2 \in \set{0,1}$ and  $m$ is minimal.

Note that $v$ is non--empty as otherwise $w$ would equal $t^{s-r}$ in $G$ and so be be the empty word by the free-by-cyclic structure of $G$.  Note also that no $v_i$ is the empty word since $m$  is minimal.  

One can obtain $t^r w$ from $vt^s$  by carrying all the $t^{\pm 1}$ to the left and freely reducing. More particularly, the $t^s$ at the end of $vt^s$ and all the $t^{\pm 1}$ in $v_m$ can be collected immediately to the left of $v_m$, and then those $t^{\pm 1}$ and the $t^{\pm 1}$ in $v_{m-1}$ can be carried to the left of $v_{m-1}$, and so on.   Accordingly, inductively define $w'_m, \ldots, w'_1$ and $r_m, \ldots, r_0$ by setting $r_m  := s$ and then for $i=m, \ldots, 1$ taking $r_{i-1}$ and $w'_i = w'_i(a_1, \ldots, a_k)$ to be the unique integer and reduced word such that  $v_i t^{r_i} = t^{r_{i-1}} w'_i$.     Then $r_0 = r$ and  $w$ is (a priori) the freely reduced form of $w'_1 \ldots  w'_m$.
We claim that, in fact, $w'_1 \ldots  w'_m$ is the partition of $w$ into pieces of rank $k$  --- that is, $m=n$ and $w'_i = w_i$ for all $i$.  This will suffice  to establish the lemma.

To prove this claim, we will show that for all $i$, if $v_i = {(a_k t)}^{\epsilon_1} \, \tau \, {(a_k t)}^{-\epsilon_2}$ where $\tau = \tau(a_1t, \ldots, a_{k-1}t)$ and $\epsilon_1, \epsilon_2 \in \set{0,1}$, then $w'_i$ is a reduced word ${a_k}^{\epsilon_1} \, \pi \, {a_k}^{-\epsilon_2}$ for some $\pi = \pi(a_1, \ldots , a_{k-1})$.  Moreover,  if $\epsilon_1 = \epsilon_2 = 0$, then $\pi$ is not the empty word.  In particular, no $w'_i$  is the empty word.  

Well, $v_i t^{r_i} = t^{r_{i-1}} w'_i$.  Consider  the process of  carrying each $t^{\pm 1}$ in $v_it^{r_i}$ to the front of the word, applying $\theta^{\pm 1}$ to each $a_j$  they pass and then freely reducing, to give $t^{r_{i-1}} w'_i$.  Throughout this process, no new ${a_k}^{\pm 1}$ are produced and, such is $\theta$, no $a_l$ appear to the left of the $a_k$ in $v_i$ (if present) or the the right of the ${a_k}^{-1}$ (if present) ---  see (\ref{theta definition}) and (\ref{inverse eqn}).  This means that the only way $w'_i$ could fail to be a reduced word of the form ${a_k}^{\epsilon_1} \, \pi \, {a_k}^{-\epsilon_2}$ where  $\pi = \pi(a_1, \ldots , a_{k-1})$, would be for $\epsilon_1$ and $\epsilon_2 $ to both be $1$ and $\pi$  be the empty word.   But in that case, $w'_i$ would be the empty word and so $v_i$ would equal $t^{r_{i-1} - r_i}$ in $G_k$ and  $r_{i-1} - r_i$ would be $0$ by Lemma~\ref{t^k in H}.  But then  $v_i$ would be the empty word by Proposition~\ref{freeness of H_k} which, as we observed, is not the case.  Likewise, when $\epsilon_1 = \epsilon_2 = 0$,  it cannot be the case that $\pi = w'_i$ is the empty word,  as otherwise $v_i$ would again be the empty word.

So properties (i), (ii) and (iii) all apply to $w'_1$, \ldots, $w'_m$ as they are inherited the corresponding properties for $v_1$, \ldots, $v_m$.    It follows from these properties together with the fact that each $w'_i$ is reduced, that $w'_1 \ldots  w'_m$ is reduced and is the partition of $w$ into pieces of rank $k$.  
\end{proof}

\section{\texorpdfstring{Passing powers of $t$ through $\theta^n({a_k}^{\pm1})$.}{Passing powers of t  through theta-terms.}} \label{Upper bound preliminaries2}

The words $\theta^n({a_k}^{\pm1})$ will play a crucial role in our proof that $\Dist^{G_k}_{H_k} \preceq A_k$.  The next lemma reveals their recursive structure.  The first part is proved by an induction on $n$.  The second part is then an immediate consequence.

\begin{lemma} \label{expanding theta}
  \begin{align*}
     \theta^n(a_k) & \ = \  \begin{cases}
     \parbox{65mm}{$a_k \;  \theta^0(a_{k-1}) \; \theta^1(a_{k-1}) \ldots \theta^{n-1}(a_{k-1})$} \quad &n>0 \\
      a_k \quad &n=0 \\
      a_k \; \theta^{-1}({a_{k-1}}^{-1}) \; \theta^{-2}({a_{k-1}}^{-1}) \ldots \theta^n({a_{k-1}}^{-1}) \quad &n<0, \\
    \end{cases} \\
    \\
    \theta^n({a_k}^{-1})  & \ = \  \begin{cases}
      \parbox{65mm}{$\theta^{n-1}({a_{k-1}}^{-1}) \; \theta^{n-2}({a_{k-1}}^{-1}) \ldots \theta^0({a_{k-1}}^{-1}) \; {a_k}^{-1}$} \quad &n>0 \\
      {a_k}^{-1} \quad &n=0 \\
      \theta^n(a_{k-1}) \; \theta^{n+1}(a_{k-1}) \ldots \theta^{-1}(a_{k-1}) \; {a_k}^{-1} \quad &n<0. \\
    \end{cases}
  \end{align*}
\end{lemma}

When attempting to carry a power of $t$ through a word $w=w(a_1, a_2, \ldots)$, we will frequently be faced with the special case where $w$ is of the form $\theta^n({a_k}^{\pm1})$.  We now focus on this situation.

\begin{defn} \label{def lambda}
  Define $$\Lambda = \bigcup_{i\in \Z} Ht^i.$$  For each integer $k \geq 1$, define $$S_k = \{n \in \Z \, : \, \theta^n(a_k) \in \Lambda\}$$  and define the function $\phi_k : S_k \rightarrow \Z$ by setting $\phi_k(n)$ to be the unique integer satisfying $$\theta^n(a_k) t^{\phi_k(n)} \in H.$$
\end{defn}

Note that this extends the previous definition of the functions $\phi_k$ given in Section~\ref{hydra functions versus Ackermann's functions} since $\phi_k(n) = \H(\theta^n(a_k))$ for $n \geq 0$.

\begin{lemma} \label{S_i}
  \mbox{}
  \begin{enumerate}
    \item $S_1 = \Z$ and $\phi_1(n) = 1$ for all $n \in S_1$. 
    \item $S_2 = \Z$ and $\phi_2(n) = n+1$ for all $n \in S_2$.
    \item If $k \geq 3$, then $S_k = \N$.
  \end{enumerate}
\end{lemma}

\begin{proof}
  It is easy to check that $S_1 = S_2 =\Z$, $\phi_1(n) = 1$, $\phi_2(n) = n+1$ and that $\N \subseteq S_k$ for all $k$.

  Let $k \geq 3$ and suppose that $n < 0$ lies in $S_k$.  Since $\theta^n(a_k) t^{\phi_k(n)}$ lies in $H$, so does $(a_k t)^{-1} \theta^n(a_k) t^{\phi_k(n)} = {a_{k-1}}^{-1} \theta^{-1}({a_{k-1}}^{-1}) \ldots \theta^{n+1}({a_{k-1}}^{-1}) t^{\phi_k(n)-1}$, and hence, by Lemma~\ref{Key Lemma}, ${a_{k-1}}^{-1}$ lies in $H t^r$ for some $r$.  It follows that $\theta^{-r}(a_{k-1}) t^r \in H$ and so $r = \phi_{k-1}(-r)$.  If $k=3$, this is contradiction, since it implies $r=-r+1$.  If $k >3$, then $-r \in S_{k-1}$, and so, by the induction hypothesis, $r \leq 0$.  But then $\phi_{k-1}(-r) \geq 1$, by \eqref{phi geq 1}, and hence $r \geq 1$, a contradiction.
\end{proof}

Let $d_H$ denote the word metric on $H$ with respect to the generating set $a_1t, a_2t, \ldots$.

\begin{lemma} \label{Lemma1}
  If $n \in S_k$ and $h = \theta^n(a_k) t^{\phi_k(n)}$, then $d_H(1, h) = \phi_k(|n|)$.
\end{lemma}

\begin{proof}
  If $k = 1$, then the result is obvious.  If $k = 2$, then $h = a_2 {a_1}^n t^{n+1} = (a_2 t) (a_1 t)^n$ so $d_H(1,h) = 1 + |n| = \phi_k(|n|)$.  If $k \geq 3$, then $n \geq 0$.  Thus the word $\theta^n(a_k)$ is positive and hence $d_H(1,h) = \phi_k(n) = \phi_k(|n|)$.
\end{proof}

\begin{lemma} \label{lem2}

  \mbox{}
  \begin{enumerate}
    \item Let $h = t^r \theta^i(a_k) t^{-s}$.  Then $h \in H$ if and only if $i-r \in S_k$ and $s=r-\phi_k(i-r)$.

    \item Let $h = t^r \theta^i({a_k}^{-1}) t^{-s}$.  Then $h \in H$ if and only if $i-s \in S_k$ and $r=s - \phi_k(i-s)$.
  \end{enumerate}
\end{lemma}

\begin{proof}
  For (i), note that $h = \theta^{i-r}(a_k) t^{r-s}$ and apply Definition~\ref{def lambda}.  For (ii), note that $h^{-1} = t^s \theta^i(a_k) t^{-r}$ and apply (i).
\end{proof}

\begin{lemma} \label{lem3}
  If $k \geq 3$ and $t^r \theta^i({a_k}^{-1}) \in \Lambda$, then  $r <i$.
\end{lemma}

\begin{proof}
  If $t^r \theta^i({a_k}^{-1}) \in H t^s$, then Lemmas~\ref{S_i} and \ref{lem2} give that $i -s \geq 0$ and $s - r = \phi_k(i-s) \geq 1$.  Thus $i-r \geq 1$.
\end{proof}

The exceptional nature of $S_1$ and $S_2$ highlighted by Lemma~\ref{S_i} means that small values of $k$ will have to be treated separately in our proof.  This motivates the inclusion of the following result, a special case of Lemma~\ref{lem2}.  Note in particular that (ii) implies that $t^r \theta^i({a_2}^{-1}) \in \Lambda$ if and only if $r+i$ is odd.

\begin{lemma} \label{lem7}
  \mbox{} \begin{enumerate}
    \item Let $h = t^r \theta^i(a_2) t^{-s}$.  Then $h \in H$ if and only if $s = 2r - i - 1$.

    \item Let $h = t^r \theta^i({a_2}^{-1}) t^{-s}$.  Then $h \in H$ if and only if $s = \frac{1}{2}(r+i+1)$.
  \end{enumerate}
\end{lemma}

\begin{proof}
  This follows immediately from Lemma~\ref{lem2} and the fact, given in Lemma~\ref{S_i}, that $\phi_2(n) = n+1$.
\end{proof}

The following result concerns passing a power of $t$ through a sequence of terms of the form $\theta^i({a_2}^{\pm1})$.  The statement is made neater by the use of the following formula, which is a consequence of Lemma~\ref{expanding theta} :  \begin{align*}
  \theta^a({a_3}^{-1}) \theta^b(a_3) \ =  \  \begin{cases}
    \theta^a(a_2) \ldots \theta^{b-1}(a_2) \quad &a<b, \\
    1 \quad &a=b, \\
    \theta^{a-1}({a_2}^{-1}) \ldots \theta^b({a_2}^{-1}) \quad &a>b. \\
  \end{cases}
\end{align*}

\begin{lemma} \label{sequence of theta(a_2) terms}
  Let $\sigma = t^r \theta^a({a_3}^{-1}) \theta^b(a_3)$ and $s = 2^{b-a} (r-a-2) + b + 2$ for some integers $r, a, b$.  Then $\sigma \in \Lambda$ if and only if $s$ is an integer.  Furthermore, in this case, $\sigma \in Ht^s$.
\end{lemma}

\begin{proof}
   We split the proof into two claims.  The first claim is that if $\sigma \in H t^{s'}$ for some integer $s'$, then $s = s'$.  In particular, this implies that if $\sigma \in \Lambda$, then $s$ is an integer.  If $a = b$, then clearly $s' = r = s$.  If $a < b$, then $\theta^a({a_3}^{-1}) \theta^b(a_3) = \theta^a(a_2) \ldots \theta^{b-1}(a_2)$.  By the  Lemma~\ref{Key Lemma}, there exist integers $r = r_0, r_1, \ldots, r_{b-a} = s'$ such that $t^{r_i} \theta^{a+i}(a_2) \in H t^{r_{i+1}}$.  By Lemma~\ref{lem7}, $r_{i+1} = 2 r_i - a - i - 1$, which solves to give $r_i = 2^i(r-a-2) + i + a + 2$.  Substituting $i=b-a$ gives $s' = s$.  On the other hand, suppose that $a > b$.  Note that $t^r \theta^a({a_3}^{-1}) \theta^b(a_3) \in H t^{s'}$ implies that $t^{s'} \theta^b({a_3}^{-1}) \theta^a(a_3) \in H t^r$.  Since $b < a$, we can substitute into the above solution to obtain $r = 2^{a-b}(s'-b-2) + a + 2$, which rearranges to give $s'= s$. This completes the proof of our first claim.

  The second claim is that if $s$ is an integer, then $\sigma \in Ht^s$.  If $a= b$, then this clearly holds.   Suppose that $a < b$.  Then $\sigma = t^r \theta^a(a_2) \ldots \theta^{b-1}(a_2)$, so certainly $\sigma \in \Lambda$ since all the letters ${a_2}^{\pm1}$ that appear are positive.  Therefore $\sigma\in H t^s$ by the first claim.  Now suppose that $a > b$.  Since $s$ is an integer, we can define $\tau = t^s \theta^b({a_3}^{-1}) \theta^a(a_3) = t^s \theta^b(a_2) \ldots \theta^{a-1}(a_2)$.  Then certainly $\tau \in \Lambda$ --- say $\tau \in H t^{r'}$.  By the first claim, $r' = 2^{a-b}(s-b-2) + a +2 = r$.  Therefore $t^s \theta^b({a_3}^{-1}) \theta^a(a_3) \in t^r$, whence $t^r \theta^a({a_3}^{-1}) \theta^b(a_3) \in Ht^s$, and the second claim is proved.
\end{proof}

\section{\texorpdfstring{An upper bound on the distortion of $H_k$ in $G_k$}{An upper bound on the distortion of Hk in Gk}} \label{DistortionUpper}

Next we turn to estimates associated with pushing a power of $t$ from left to right through a word  $w=w(a_1, \ldots, a_k)$ or through a piece of $w$, so as to leave a word on $a_1t, \ldots, a_kt$ times a power of $t$.  We will need to keep track of both the length of that word on the $a_1t, \ldots, a_kt$ and the power of $t$ that emerges to its  right.    Accordingly, let us define four families of functions, $\psi_{k ,l}(n)$, $\Psi_{k,l, p}(n)$, $\kappa_{k,l}(n)$, $K_{k,l, p}(n)$ for integers $k \geq 1$ and $l, p,n \geq 0$.

\begin{itemize}
\item  $\psi_{k, l} (n)$ is the least integer $N$ such that if $h  \in H$  is represented by a  word $t^r \pi t^{-s}$ with $\pi$ a piece of rank $k$,  with $\ell(\pi) \leq l$, and with $\abs{r} \leq n$, then $d_H(1,h) \leq N$.

\item  \rule{0mm}{6mm}$\Psi_{k, l, p} (n)$ is the least integer $N$ such that if $h  \in H$  is represented by a  word $t^r w t^{-s}$ with $w=w(a_1, \ldots, a_k)$ a word of at most $p$ pieces, with $\ell(w) \leq l$, and with $\abs{r} \leq n$, then $d_H(1,h) \leq N$.

\item  \rule{0mm}{6mm}$\kappa_{k, l} (n)$ is the least integer $N$ such that if $\pi$ is a piece of rank $k$ with  $\ell(\pi) \leq l$   and $r$ is an integer with $\abs{r} \leq n$ and $t^r \pi \in \Lambda$, then  $ t^r \pi \in Ht^s$ for some $s$ with $\abs{s} \leq N$.

\item  \rule{0mm}{6mm}$K_{k, l, p} (n)$ is the least integer $N$ such that if $w$ is a  word of rank at most $k$ with at most $p$ pieces and with $\ell(w) \leq l$  and $r$ is an integer with $\abs{r} \leq n$ and $t^r w \in \Lambda$, then  $ t^r w \in Ht^s$ for some $s$ with $\abs{s} \leq N$.
\end{itemize}

We will frequently make use, without further comment, of the fact that each of these functions is increasing in $k$, $l$, $p$ and $n$.  

The main technical result of this section is the following proposition.  In the corollary that follows it we explain how the upper bound it gives on $\Psi_{k,l, p}(n)$ leads to our desired bound $\Dist^{G_k}_{H_k} \preceq A_k$.

\begin{prop} \label{upper bounds}
  For all $k \geq 1$, there exist integers $C_k \geq 1$ such that for all $l,p,n \geq 0$, 
  \begin{align*}
    \kappa_{k,l}(n) \ &\leq \ A_{k-1}(C_k n + C_k l), \\
    K_{k,l, p}(n) \ &\leq \ {A_{k-1}}^{(p)}(C_k n + C_k l), \\
    \rule{0mm}{5mm} \psi_{k,l}(n) \ &\leq \ A_{k-1}(C_kn + C_k l), \label{Bound on psi} \\
    \Psi_{k,l,p}(n) \ &\leq \ {A_{k-1}}^{(3p)}(C_kn + C_k l).
  \end{align*}
\end{prop}

\begin{cor}  \label{the upper bound at last}
  For all $k \geq 1$, the distortion function of $H_k$ in $G_k$ satisfies \begin{equation*}
    Dist_{H_k}^{G_k} \preceq A_k.
  \end{equation*}
\end{cor}

\begin{proof}[Proof of Corollary~\ref{the upper bound at last}.]
  Since $G_1 \cong \Z^2$ and $H_1 \cong \Z$, $H_1$ is undistorted in $G_1$ and $Dist_{H_1}^{G_1} \preceq A_1$.  Now suppose that $k \geq 2$ and that $u=u(a_1, \ldots, a_k, t)$ is a word of length at most $n$ representing an element of $H$.  By carrying each $t^{\pm 1}$ to the front, we see that $u$ is equal in $G_k$ to $t^r w$ for some integer $r$ and some freely reduced word $w = w(a_1, \ldots a_k)$.  These satisfy $\abs{r} \leq n$ and $\ell(w) \leq C n^k$ for some integer $C>0$ depending only on $k$ --- see, for example, Section~3.3 of \cite{Bridson2}.

  We first show that the number of pieces of $w$ is at most $n+1$.  Indeed, the process of carrying each $t^{\pm1}$ to the front of $u$ has the effect of applying $\theta^{\pm1}$ to each $a_i$ it passes.  The form of the automorphism $\theta$ ensures that no new ${a_k}^{\pm1}$ are created by this process.  The number of occurrences of ${a_k}^{\pm1}$ in $w$, which we denote by $\ell_k(w)$, is therefore at most $n$.  Let $w = w_1 \ldots w_p$ be the partition of $w$ into pieces.  Say $w_i = {a_k}^{\epsilon_i^-} \pi_i {a_k}^{-\epsilon_i^+}$ where $\epsilon_i^-, \epsilon_i^+ \in \{ 0, 1\}$ and $\pi_i = \pi_i(a_1, \ldots, a_{k-1})$.  Observe that, for each $i$, precisely one of $\epsilon_i^+$ and $\epsilon_{i+1}^-$ is equal to $1$.   Indeed, if $\epsilon_i^+=\epsilon_{i+1}^-=0$, then the pieces $w_i$ and $w_{i+1}$ could be concatenated to form a single piece, contradicting the minimality of $p$, and if $\epsilon_i^+=\epsilon_{i+1}^-=1$, then $w$ would not be freely reduced.  So $$\ell_k(w) \ =  \ \sum_{i=1}^p (\epsilon_i^- + \epsilon_i^+) \ = \ \epsilon_1^- + \sum_{i=1}^{p-1} (\epsilon_i^+ + \epsilon_{i+1}^-) + \epsilon_n^+ \ = \  \epsilon_1^- + p-1+ \epsilon_n^+,$$ whence $p \leq \ell_k(w) +1 \leq n+1$.

  Now, $$d_H(1,u) \ = \  d_H(1, t^rw) \  \leq \  \Psi_{k, \ell(w), p}(|r|) \ \leq \ \Psi_{k, Cn^k, n+1}(n),$$ which is at most $${A_{k-1}}^{(3n+3)}(C_k C n^k + C_k n)$$ by Proposition~\ref{upper bounds}.  
  Choose an integer $N$ large enough that $n^k \leq 2^n$ for $n \geq N$.  
  Then, for $n \geq \max \{N, 1\}$, \begin{align*}
    d_H(1,u) \ &\leq \ {A_{k-1}}^{(3n+3)}(C_kCA_2(n) + C_k n) &&\text{by \eqref{A_k5}} \\
    &\leq \ {A_{k-1}}^{(3n+3)}(C_kCA_k(n) + C_k n) &&\text{by \eqref{A_k7}, \eqref{A_k5}} \\
    &\leq \ {A_{k-1}}^{(3n+3)}(A_k(C_kCn) + C_k n) &&\text{by \eqref{A_k5},  \eqref{A_k1}} \\
    &\leq \ {A_{k-1}}^{(3n+3)}(A_k((C_kC + C_k)n)) &&\text{by \eqref{A_k5}, \eqref{A_k9}} \\
    &= \ A_k((C_kC + C_k + 3)n + 3) &&\text{by \eqref{A_k4}}.
  \end{align*}
\end{proof}

Proposition~\ref{upper bounds} will follow from the relationships between $\psi_{k ,l}(n)$, $\Psi_{k,l,p}(n)$, $\kappa_{k,l}(n)$ and $K_{k,l,p}(n)$ set out in the next proposition.  Of its claims, \eqref{kappa_k} and \eqref{psik} are the most challenging to establish; we postpone their proof to Proposition~\ref{Bound on kappa and psi}, which itself will draw  on  Lemmas~\ref{Overcoming a_k}, \ref{Cancelling at either end} and \ref{ell(v) geq ell(u)}.  

\begin{prop}\label{exponent and length bounds}
For integers $k \geq 1$  and  $l,p,n \geq 0$,
\begin{align}
 \kappa_{1,l}(n) \ &\leq \ n+1, \label{kappa1} \\
 K_{k,l,p}(n) \ &\leq \ \max_{\substack{q \leq p \\ l_1 + \ldots + l_q \leq l}} \left\{\kappa_{k,l_1}( \ldots \kappa_{k, l_{q-1}}(\kappa_{k, l_q}(n)) \ldots ) \right\}, \label{K_k}  \\
 \kappa_{k+1,l}(n) \ &\leq \ 2K_{k,l,l}(2\phi_{k+1}(n)), \label{kappa_k}  \\
 \rule{0mm}{5mm}\psi_{1,l}(n) \ &\leq \ 1, \label{psi1} \\
 \Psi_{k,l,p} (n) \ &\leq \ p\psi_{k,l}(K_{k,l,p}(n)), \label{Psik}    \\
 \psi_{k+1,l}(n) \ &\leq  \ 3K_{k,l,l}(2\phi_{k+1}(n)) + \Psi_{k,l,l}(2\phi_{k+1}(n)). \label{psik}
 \end{align}
\end{prop}

\begin{proof}
 We first establish~\eqref{kappa1} and \eqref{psi1}.  Consideration of the empty word gives that $\kappa_{k,0}(n) = n$ and $\psi_{k,0}=0$.  Now suppose that $l \geq 1$ and note that the only pieces of rank $1$ are ${a_1}^{\pm1}$.  If $h = t^r {a_1}^{\pm1} t^{-s}$ lies in $H$, then $d_H(1,h) = 1$ and $r-s = \pm1$, whence $|s| \leq |r| + 1$.  Thus $\kappa_{k,l}(n) \leq n+1$ and $\psi_{k,l}(n) = 1$.

For \eqref{K_k} and \eqref{Psik}, let $h = t^r w t^{-s}$ where $w=w(a_1, \ldots, a_k)$ is a word of length at most $l$ with at most $p$ pieces and $|r| \leq n$.  Let $w = w_1 \ldots w_q$ be the partition of $w$ into pieces, where $q \leq p$.  If $h \in H$, then Lemma~\ref{Key Lemma} implies that there exist integers $r=r_0, r_1, \ldots, r_q = s$ and elements $h_1, \ldots, h_q$ in $H$ such that $t^{r_{i-1}} w_i = h_i t^{r_i}$.  Thus $|r_i| \leq \kappa_{k, \ell(w_i)}(|r_{i-1}|)$, whence $$|s|  \ \leq  \ \kappa_{k, \ell(w_q)}( \ldots (\kappa_{k, \ell(w_1)}(|r|)) \ldots ) \  \leq \  \kappa_{k, \ell(w_q)}( \ldots (\kappa_{k, \ell(w_1)}(n)) \ldots )$$ and we obtain inequality~(\ref{K_k}).  For inequality~(\ref{Psik}), note that $|r_i| \leq K_{k, \ell(w_1 \ldots w_i), i}(|r|) \leq K_{k,l,p}(n)$, whence $$d_H(1,h) \ \leq \  \sum_{i=1}^{q} d_H(1, h_i) \ \leq  \ \sum_{i=1}^{q} \psi_{k, \ell(w_i)}(|r_{i-1}|) \ \leq \ p \psi_{k,l}(K_{k,l,p}(n)).$$

Finally, \eqref{kappa_k} and \eqref{psik} will follow from Proposition~\ref{Bound on kappa and psi}.
\end{proof}

We now derive Proposition~\ref{upper bounds} from Proposition~\ref{exponent and length bounds}.   We first  use \eqref{kappa1}, \eqref{K_k} and \eqref{kappa_k} to obtain bounds on $\kappa_{k,l}(n)$ and $K_{k,l,p}(n)$ in terms of Ackermann's functions.   We then derive bounds on $\psi_{k ,l}(n)$ and $\Psi_{k,l,p}(n)$  from (\ref{psi1}), (\ref{Psik}) and (\ref{psik}), having fed in our bounds on $\kappa_{k,l}(n)$ and $K_{k,l,p}(n)$.

\begin{proof}[Proof of Proposition~\ref{upper bounds}]
   We will need the inequality, established in Lemma~\ref{phi_k vs Ackermann}, that for $n \geq 0$ and $k \geq 2$, \begin{equation}\phi_k(n) \leq A_{k-1}(n+k). \label{phik}\end{equation}

We first prove that there exist integers $D_k \geq 1$ such that \begin{align}
    \kappa_{k,l}(n) \ &\leq \ A_{k-1}(D_k n + D_k l), \label{Bound on kappa}\\
    K_{k,l, p}(n) \ &\leq \ {A_{k-1}}^{(p)}(D_k n + D_k l). \label{Bound on K}
   \end{align}  Inequalities~\eqref{kappa1} and \eqref{K_k} together imply that $K_{1,l,p}(n) \leq n + p$.  Thus \eqref{Bound on kappa} and \eqref{Bound on K} hold in the case $k=1$ with $D_1 = 1$.  Now suppose that $k \geq 2$ and that \eqref{Bound on kappa} and \eqref{Bound on K} hold for smaller values of $k$.  If $l=0$, then, using \eqref{A_k2}, we calculate that $\kappa_{k,l}(n) = n \leq A_{k-1}(n)$.  If $l \geq 1$, then \begin{align*}
    \kappa_{k,l}(n) \ &\leq \ 2K_{k-1,l,l}(2 \phi_k(n)) &&\text{by \eqref{kappa_k}}  \\
    &\leq \ 2K_{k-1,l,l}(2A_{k-1}(n+k)) &&\text{by \eqref{phik}}\\
    &\leq \ 2{A_{k-2}}^{(l)}(2D_{k-1}A_{k-1}(n+k) + D_{k-1}l) &&\\
    &\leq \ 2{A_{k-2}}^{(l)}(A_{k-1}(2D_{k-1}n + D_{k-1}l + 2D_{k-1}k)) &&\text{by \eqref{A_k5}, \eqref{A_k1}, \eqref{A_k9}}\\
    &= \ 2A_{k-1}(2D_{k-1}n + (D_{k-1}+1)l + 2D_{k-1}k) &&\text{by \eqref{A_k4}}\\
    &\leq \ A_{k-1}(4D_{k-1}n + 2(D_{k-1}+1)l + 4D_{k-1}k) &&\text{by \eqref{A_k1}} \\
    &\leq \ A_{k-1}(4 D_{k-1}n + [2(D_{k-1}+1) + 4 D_{k-1}k]l) &&\text{by \eqref{A_k5}}.
  \end{align*}  Taking $D_k = \max \{ 2(D_{k-1}+1) + 4D_{k-1}k,1 \}$, we obtain \eqref{Bound on kappa}.

  For \eqref{Bound on K} we calculate that \begin{align*}
    K_{k,l,p}(n) \ &\leq \ \max_{\substack{q \leq p \\ l_1 + \ldots + l_q \leq l}} \left\{\kappa_{k,l_1}( \ldots \kappa_{k,l_{q-1}}(\kappa_{k, l_q}(n)) \ldots ) \right\} &&\text{by \eqref{K_k}}\\
    &\leq \ \max_{\substack{q \leq p \\ l_1 + \ldots + l_q \leq l}} \left\{A_{k-1}( \ldots A_{k-1}(A_{k-1}(D_kn + D_k l_q) + D_k l_{q-1}) \ldots )\right\} &&\text{by \eqref{A_k5}}\\
    &\leq \ \max_{\substack{q \leq p \\ l_1 + \ldots + l_q \leq l}} \left\{ {A_{k-1}}^{(q)}  \left( D_kn + D_k \sum_{i=1}^q l_i \right)   \right\} &&\text{by \eqref{A_k5}, \eqref{A_k9}}\\
    &\leq \ \max_{\substack{q \leq p}} \left\{ {A_{k-1}}^{(q)}(D_kn + D_kl) \right\} &&\text{by \eqref{A_k5}}\\
    &\leq \ {A_{k-1}}^{(p)}(D_kn + D_kl) &&\text{by \eqref{A_k2}}.
  \end{align*}

  Next, we combine \eqref{psi1}, \eqref{Psik} and \eqref{psik} with \eqref{Bound on kappa} and \eqref{Bound on K} to deduce that there exist integers $E_k, F_k \geq 1$ such that \begin{align}
    \psi_{k,l}(n) \ &\leq \ A_{k-1}(E_kn + E_k l), \label{Bound on psi} \\
    \Psi_{k,l,p}(n) \ &\leq \ {A_{k-1}}^{(3p)}(F_kn + F_k l). \label{Bound on Psi}
  \end{align}  It follows from \eqref{psi1} and \eqref{Psik} that $\Psi_{1,l,p}(n) \leq p$.  Thus \eqref{Bound on psi} and \eqref{Bound on Psi} hold in the case $k=1$ with $E_k = F_k = 1$.  Now suppose that $k \geq 2$ and that \eqref{Bound on psi} and \eqref{Bound on Psi} hold for smaller values of $k$.  If $l=0$, then $\psi_{k,l}(n) = 0 \leq A_{k-1}(0)$.  If $l \geq 1$, then \begin{align*}
    \psi_{k,l}(n) \ &\leq \ 3 K_{k-1,l,l}(2\phi_k(n)) + \Psi_{k-1,l,l}(2\phi_k(n)) &&\text{by \eqref{psik}} \\
    &\leq \ 3 K_{k-1,l,l}(2A_{k-1}(n+k)) + \Psi_{k-1,l,l}(2A_{k-1}(n+k)) &&\text{by \eqref{phik}} \\
    &\leq \ 3 {A_{k-2}}^{(l)} (2D_{k-1} A_{k-1}(n+k) + D_{k-1}l)  \\ &\mspace{50mu} + {A_{k-2}}^{(3l)} (2 F_{k-1} A_{k-1}(n+k) + F_{k-1}l) &&\text{by \eqref{Bound on K}} \\
    &\leq \ 3 {A_{k-2}}^{(l)} (A_{k-1}(2D_{k-1} (n+k) + D_{k-1}l))  \\ &\mspace{50mu} + {A_{k-2}}^{(3l)} (A_{k-1}(2 F_{k-1} (n+k) + F_{k-1}l)) &&\text{by \eqref{A_k5}, \eqref{A_k1}, \eqref{A_k9}} \\
    &= \ 3 A_{k-1}(2D_{k-1} (n+k) + (D_{k-1}+1)l)  \\ &\mspace{50mu} + A_{k-1}(2 F_{k-1} (n+k) + (F_{k-1}+3)l) &&\text{by \eqref{A_k4}} \\
    &\leq \ A_{k-1}(6D_{k-1} (n+k) + 3(D_{k-1}+1)l)  \\ &\mspace{50mu} + A_{k-1}(2 F_{k-1} (n+k) + (F_{k-1}+3)l) &&\text{by \eqref{A_k1}} \\
    &\leq \ A_{k-1} (2(3D_{k-1}+F_{k-1})(n+k) + (3D_{k-1}+F_{k-1}+4)l) &&\text{by \eqref{A_k3}} \\
    &\leq \ A_{k-1} (2(3D_{k-1}+F_{k-1})n + (3(2k+1)D_{k-1}+(2k+1)F_{k-1}+4)l). &&
  \end{align*}  
 Taking $E_k = 3(2k+1)D_{k-1}+(2k+1)F_{k-1}+4$, we obtain \eqref{Bound on psi}.

 If $p=0$ or $l=0$, then, using \eqref{A_k2}, we calculate that $\Psi_{k,l,p}(n) = 0 \leq {A_{k-1}}^{(3p)}(0)$.  If $l, p \geq 1$, then \begin{align*}
    \Psi_{k,l,p}(n) \ &\leq \ p\psi_{k,l}(K_{k,l,p}(n)) &&\text{by \eqref{Psik}} \\
    &\leq \ p \psi_{k,l}({A_{k-1}}^{(p)} (D_k n + D_k l)) &&\text{by \eqref{Bound on K}} \\
    &\leq \ p A_{k-1}(E_k {A_{k-1}}^{(p)} (D_k n + D_k l) + E_kl) && \\
    &\leq \ p {A_{k-1}}^{(p+1)} (D_k E_k n + (D_k + 1)E_k l) &&\text{by \eqref{A_k5}, \eqref{A_k2}, \eqref{A_k1}, \eqref{A_k9}} \\
    &\leq \ {A_{k-1}}^{(2p+1)} (D_k E_k n + (D_k + 1)E_k l) &&\text{by \eqref{A_k6}}, \\
    &\leq \ {A_{k-1}}^{(3p)} (D_k E_k n + (D_k + 1)E_k l) &&\text{by \eqref{A_k2}.}
  \end{align*}  Taking $F_k = (D_k + 1)E_k$, we obtain \eqref{Bound on Psi}.

  Finally, the proof is completed by taking $C_k = \max\{D_k, E_k, F_k\}$ and applying \eqref{A_k5}.
\end{proof}

The remainder of this section is devoted to establishing \eqref{kappa_k} and \eqref{psik}.  This is done in Proposition~\ref{Bound on kappa and psi}, which draws on  Lemmas~\ref{Overcoming a_k}, \ref{Cancelling at either end} and \ref{ell(v) geq ell(u)} that follow.   We now outline our strategy.

Suppose that $t^{r} {a_k}^{\epsilon_1} w {a_k}^{-\epsilon_1} t^{-s}$, where $r,s \in \Z$, $\epsilon_1, \epsilon_2 \in \{0,1\}$ and $w=w(a_1, \ldots, a_{k-1})$, represents an element $h \in H$.  Our approach will be to find elements $h_1, h_2 \in H$, integers $r', s'$ and a word $w'=w'(a_1, \ldots, a_{k-1})$ such that $h$ is represented by $h_1 t^{r'} w' t^{-s'} h_2$.  The functions $K_{k-1,\ast,\ast}$ and $\Psi_{k-1, \ast, \ast}$ will then control the behaviour of the subword $t^{r'} w' t^{-s'}$.  Together with estimates for $d_H(1, h_i)$, $|r'|$, $|s'|$ and $\ell(w')$, this will allow us to derive bounds on $|s|$ and $d_H(1, h)$.

As indicated by Lemma~\ref{S_i}, the case $k=2$ is exceptional and so will be treated separately.  For $k \geq 3$, the $h_1$, $h_2$ $r'$, $s'$ and $w$ will be produced by Lemma~\ref{Overcoming a_k}.  This lemma takes integers $k$, $n$ and $\epsilon$, with $k \geq 3$ and $\epsilon \in \{0,1\}$, and gives an integer $n'$, an element $h \in H$ and a word $u = u(a_1, \ldots, a_{k-1})$ such that $t^n a_k = h t^{n'} u$ in $G$.  Applying Lemma~\ref{Overcoming a_k} to $k$, $r$ and $\epsilon_1$ will produce $r'$, $h_1$ and a word $u_1$.  Applying Lemma~\ref{Overcoming a_k} to $k$, $s$ and $\epsilon_2$ will produce $s'$, ${h_2}^{-1}$ and a word $u_2$.  The word $w'$ will then be defined to be the free reduction of  $\tilde{w} := u_1 w {u_2}^{-1}$.

The relationship between the input and output of Lemma~\ref{Overcoming a_k}  is determined by which of the following holds:
\begin{enumerate}
\item $\epsilon =0$, 
\item $\epsilon = 1$ and $n \leq 0$, or
\item  $\epsilon = 1$ and $n > 0$.  
\end{enumerate} \emph{A priori}, this would lead to us having to consider nine distinct cases, depending on the values of $\epsilon_1$ and $\epsilon_2$ and the signs of $r$ and $s$.  To streamline the process, Lemma~\ref{Overcoming a_k} packages (i) and (ii) together: it considers the cases that either $n\epsilon \leq 0$ or $n\epsilon>0$.  As such, we need now only consider four cases, depending on the signs of $r\epsilon_1$ and $s\epsilon_2$.

The form of  $\tilde{w}$ will depend on which of (i), (ii) or (iii) applies to $r$ and $\epsilon_1$ and to $s$ and $\epsilon_2$.  Lemmas~\ref{Cancelling at either end} and \ref{ell(v) geq ell(u)} will be brought to bear to ensure that enough cancellation occurs to obtain a sufficiently strong bound on $\ell(w')$.

\begin{prop} \label{Bound on kappa and psi}
  Let $h = t^r {a_k}^{\epsilon_1} w {a_k}^{-\epsilon_2} t^{-s}$ where $\epsilon_1, \epsilon_2 \in \{0, 1\}$, $w = w(a_1, \ldots, a_{k-1})$ and $k \geq 2$.  Let $n$ and $l$ be integers with $|r| \leq n$ and $\ell(w) \leq l$.  If $h \in H$, then \begin{align*}
    |s| \ &\leq \ 2 K_{k-1, l, l}(2 \phi_k(n)), \\
    d_H(1,h) \ &\leq \ 3 K_{k-1,l,l}(2\phi_k(n)) + \Psi_{k-1, l, l}(2 \phi_k(n)).
  \end{align*}
\end{prop}

\begin{proof}
  We claim that there exist $h_1, h_2 \in H$, $r', s' \in \Z$ and $w' = w'(a_1, \ldots, a_{k-1})$ such that $h = h_1 t^{r'} w' t^{-s'} h_2$ in $G$ and \begin{align}
    |r'| \ &\leq \ 2\phi_k(n), \label{Bound on r'}\\
    |s| \ &\leq \ |s'| + 1, \label{Bound on s}\\
    d_H(1, h_1) \ &\leq \ |r'| + 1, \label{Bound on h_1}\\
    d_H(1, h_2) \ &\leq \ |s'| + 1, \label{Bound on h_2} \\
    \ell(w') \ &\leq \ l. \label{Bound on w'}
  \end{align}

  The result follows from the claim by direct calculation.  Indeed, since the number of pieces of a word is bounded by its length, \begin{align}
    |s'| \ &\leq \ K_{k-1, \ell(w'), \ell(w')}(|r'|), \label{Bound on s'} \\
    d_H(1, t^{r'} w' t^{-s'}) \ &\leq \ \Psi_{k-1, \ell(w'), \ell(w')}(|r'|). \label{Bound on h'}
  \end{align}  We will also need the inequality \begin{equation}
  K_{k,l,p}(n) \ \geq \ n,  \label{K geq n}
\end{equation}  which follows immediately from consideration of the empty word.  We can now calculate that \begin{align*}
    |s| \ &\leq \ |s'| + 1 &&\text{by \eqref{Bound on s}} \\
    &\leq \ K_{k-1, \ell(w'), \ell(w')}(|r'|) + 1 &&\text{by \eqref{Bound on s'}} \\
    &\leq \ K_{k-1, l, l}(2\phi_k(n)) + 1 &&\text{by \eqref{Bound on r'}, \eqref{Bound on w'}} \\
    &\leq \ 2K_{k-1, l, l}(2\phi_k(n)) &&\text{by \eqref{phi geq 1}, \eqref{K geq n}}\\
    \rule{0mm}{5mm}d_H(1, h) \ &\leq \ d_H(1, h_1) + d_H(1, t^{r'} w' t^{-s'}) + d_H(1, h_2) &&\\
    &\leq \ |r'| + 1 + \Psi_{k-1, \ell(w'), \ell(w')}(|r'|) + |s'| + 1 &&\text{by \eqref{Bound on h_1}, \eqref{Bound on h_2}, \eqref{Bound on h'}} \\
    &\leq \ 2\phi_k(n) + 1 + \Psi_{k-1, l, l}(2\phi_k(n)) + K_{k-1, \ell(w'), \ell(w')}(|r'|) + 1 &&\text{by \eqref{Bound on r'}, \eqref{Bound on w'}, \eqref{Bound on s'}} \\
    &\leq \ 4\phi_k(n) + \Psi_{k-1, l, l}(2\phi_k(n)) + K_{k-1, l, l}(2\phi_k(n)) &&\text{by \eqref{phi geq 1}, \eqref{Bound on r'}, \eqref{Bound on w'}} \\
    &\leq \ 3 K_{k-1, l, l}(2\phi_k(n)) + \Psi_{k-1, l, l}(2\phi_k(n)) &&\text{by \eqref{K geq n}.}
  \end{align*}

  We first prove the claim for $k=2$.  Since $t^p a_2 = (a_2 t) (a_1 t)^{-p} t^{2p-1}$, we can take $w'$ to be $w$ and define $h_1$, $h_2$, $r'$ and $s'$ by \begin{alignat*}{4}
    h_1 &= \left\{\begin{aligned}
      &1 \\
      &(a_2 t) (a_1 t)^{-r} \\
    \end{aligned}\right.&
    \hspace{6pt}&\begin{aligned}
      &\epsilon_1 = 0, \\
      &\epsilon_1 = 1, \\
    \end{aligned}&
    \hspace{25pt} r' &= \left\{\begin{aligned}
      &r \\
      &2r-1 \\
    \end{aligned}\right.&
    \hspace{6pt}&\begin{aligned}
      &\epsilon_1 = 0, \\
      &\epsilon_1 = 1, \\
    \end{aligned} \\[8pt]
     h_2 &= \left\{\begin{aligned}
      &1 \\
      &(a_1 t)^{-s} (a_2 t)^{-1} \\
    \end{aligned}\right.&
    &\begin{aligned}
      &\epsilon_2 = 0, \\
      &\epsilon_2 = 1, \\
    \end{aligned}&
    \hspace{30pt} s' &= \left\{\begin{aligned}
      &s \\
      &2s-1 \\
    \end{aligned}\right.&
    &\begin{aligned}
      &\epsilon_2 = 0, \\
      &\epsilon_2 = 1. \\
    \end{aligned}
  \end{alignat*}  Inequalities~\eqref{Bound on s} and \eqref{Bound on w'} are immediate.  For \eqref{Bound on r'}, use the fact, from Lemma~\ref{S_i}, that $\phi_2(n) = n+1$.  Inequality~\eqref{Bound on h_1} is immediate if $\epsilon_1 = 0$.  If $\epsilon_1 = 1$, then $r = \frac{1}{2}(r'+1)$, whence $|r| \leq \frac{1}{2}(|r'| + 1)$.  But $r' \neq 0$, so $|r| \leq |r'|$ and $d_H(1, h_1) = |r|+1 \leq |r'| + 1$.  Inequality~\eqref{Bound on h_2} is derived similarly.

We now prove the claim for $k \geq 3$.  First apply Lemma~\ref{Overcoming a_k} to $k, r, \epsilon_1$ to produce $r'$, $h_1$ and a word $u_1$.  Then apply it to $k,s, \epsilon_2$ to produce $s'$, ${h_2}^{-1}$ and a word $u_2$.  Defining $\tilde{w} := u_1 w {u_2}^{-1}$, we have that $h$ is represented by $h_1 t^{r'} \tilde{w} t^{-s'} h_2$ and hence that $t^{r'} \tilde{w} t^{-s'} \in H$.  It is immediate from the bounds given in Lemma~\ref{Overcoming a_k} that \eqref{Bound on r'}--\eqref{Bound on h_2} hold.  Finally, we define $w'$ to be the free reduction of $\tilde{w}$.  To establish~\eqref{Bound on w'}, we consider four cases.

 \textit{Case: $r\epsilon_1 \leq 0$, $s\epsilon_2 \leq 0$}.  We have that $\tilde{w} = w$ and so it is immediate that $\ell(w') \leq \ell(w)$.

\textit{Case: $r\epsilon_1 > 0$, $s\epsilon_2 \leq 0$}.  We have that $\tilde{w} =  \theta^{r-1}({a_{k-1}}^{-1}) \ldots \theta^0({a_{k-1}}^{-1}) w$.  Since $t^{r'} \theta^{r-1}({a_{k-1}}^{-1})$ does not lie in $\Lambda$, applying Lemma~\ref{Key Lemma} to $t^{r'} \tilde{w} t^{-s'}$ shows that, when $\tilde{w}$ is freely reduced, each ${a_{k-1}}^{-1}$ in $\theta^{r-1}({a_{k-1}}^{-1}) \ldots \theta^0({a_{k-1}}^{-1})$ cancels into $w$.  It follows from Lemma~\ref{ell(v) geq ell(u)} that $\ell(w') \leq \ell(w)$. 

 \textit{Case: $r\epsilon_1 \leq 0$, $s\epsilon_2 > 0$}.  We have that $\tilde{w} = w \theta^0(a_{k-1}) \ldots \theta^{s-1}(a_{k-1})$.  Since $t^{s'} \theta^{s-1}({a_{k-1}}^{-1})$ does not lie in $\Lambda$, applying Lemma~\ref{Key Lemma} to $t^{s'} {\tilde{w}}^{-1} t^{-r'} \in H$ shows that, when $\tilde{w}$ is freely reduced, each $a_{k-1}$ in $\theta^0(a_{k-1}) \ldots \theta^{s-1}(a_{k-1})$ cancels into $w$.  It follows from Lemma~\ref{ell(v) geq ell(u)} that $\ell(w') \leq \ell(w)$. 

\textit{Case: $r\epsilon_1 > 0$, $s\epsilon_2 > 0$}.  We have that $\tilde{w} = \theta^{r-1}({a_{k-1}}^{-1}) \ldots \theta^0({a_{k-1}}^{-1}) w \theta^0(a_{k-1}) \ldots \theta^{s-1}(a_{k-1})$.  Since neither $t^{r'} \theta^{r-1}({a_{k-1}}^{-1})$ nor $t^{s'} \theta^{s-1}({a_{k-1}}^{-1})$ lies in $\Lambda$, we are in a position to apply Lemma~\ref{Cancelling at either end}.  If case~(i) of Lemma~\ref{Cancelling at either end} occurs, then, when $\tilde{w}$ is freely reduced, each ${a_{k-1}}^{-1}$ in $\theta^{r-1}({a_{k-1}}^{-1}) \ldots \theta^0({a_{k-1}}^{-1})$ and each $a_{k-1}$ in $\theta^0(a_{k-1}) \ldots \theta^{s-1}(a_{k-1})$ cancels into $w$. Applying Lemma~\ref{ell(v) geq ell(u)} gives that $\ell(w') \leq \ell(w)$.  On the other hand, suppose that case~(ii) of Lemma~\ref{Cancelling at either end} occurs, so $w'$ is the free reduction of $\theta^{r-1}({a_{k-1}}^{-1}) \theta^{s-1}(a_{k-1})$. We will show that $r = s$, whence $w'$ is the empty word and trivially $\ell(w') \leq l$.  If $k = 3$, then $t^{r'} w' t^{-s'} = t^{r-1} \theta^{r-1}({a_2}^{-1}) \theta^{s-1}(a_2) t^{1-s} = t^{r-s} {a_1}^{s-r}$ in $G$.  Since this element lies in $H$, $r-s = s-r$, whence $r=s$. If $k=4$, then $t^{r'} w' t^{-s'}$ is freely equal to $t^{r-1} \theta^{r-1}({a_3}^{-1}) \theta^{s-1}(a_3) t^{1-s}$.  Since this lies in $H$, applying Lemma~\ref{sequence of theta(a_2) terms} and solving the resulting equation gives $r=s$.  Finally, suppose that $k > 4$.  Lemma~\ref{expanding theta} gives that \begin{align*}
  t^{r'} w' t^{-s'} \ \FreeEq  \  \begin{cases}
    t^{r-1} \theta^{r-1}(a_{k-2}) \ldots \theta^{s-2}(a_{k-2}) t^{1-s} \quad &r<s, \\
    t^{r-s} \quad &r=s, \\
    t^{r-1} \theta^{r-2}({a_{k-2}}^{-1}) \ldots \theta^{s-1}({a_{k-2}}^{-1}) t^{1-s} \quad &r>s. \\
  \end{cases} \end{align*}  By Lemma~\ref{lem3}, neither $t^{r-1} \theta^{r-2}({a_{k-2}}^{-1})$ nor $t^{s-1} \theta^{s-2}({a_{k-2}}^{-1})$ lies in $\Lambda$, since $k-2 \geq 3$.  Thus, by Lemma~\ref{Key Lemma}, both $r < s$ and $s > r$ lead to a contradiction.  Hence $r=s$ as required.
\end{proof}

\begin{lemma} \label{Overcoming a_k}
  Given integers $k,n, \epsilon$, with $k \geq 3$ and $\epsilon \in \{0,1\}$, there exists an integer $n'$, an element $h \in H$ and a word $u = u(a_1, \ldots, a_{k-1})$ such that $t^n {a_k}^\epsilon = h t^{n'} u$ in $G$, \begin{align*}
    |n|-1 \ &\leq \ |n'| \ \leq \ 2 \phi_k(|n|)\text{,  \ and} \\
    d_H(1, h) \ &\leq \ \max\{|n'|, 1\}.
  \end{align*}  Furthermore, \begin{enumerate}
    \item if $n \epsilon \leq 0$, then $u$ is the empty word;

    \item if $n \epsilon > 0$, then $n' = n-1$, $u = \theta^{n-1}({a_{k-1}}^{-1}) \ldots \theta^0({a_{k-1}}^{-1})$ and $t^{n'} \theta^{n-1}({a_{k-1}}^{-1}) \notin \Lambda$.
  \end{enumerate}
\end{lemma}

\begin{proof}
  We consider three cases.

\textit{Case: $\epsilon= 0$. } We trivially obtain an instance of conclusion (i) by taking $n' = n$, $h=1$ and $u$ to be the empty word. The upper bound on $|n'|$ follows from \eqref{phi geq 1} and \eqref{phi geq n}. 

\textit{Case: $\epsilon = 1$ and $n \leq 0$.} Following the calculation $$t^n a_k \ = \  \theta^{-n}(a_k) t^n \ = \  \theta^{-n}(a_k) t^{\phi_k(|n|)} t^{n-\phi_k(|n|)},$$ we obtain an instance of conclusion (i) by taking $n' = n-\phi_k(|n|)$, $h = \theta^{-n}(a_k) t^{\phi_k(|n|)}$ and $u$ to be the empty word.  It follows immediately from the definition of the function $\phi_k$ that $h \in H$ and from Lemma~\ref{Lemma1} that $d_H(1, h) = \phi_k(|n|)$.  By \eqref{phi geq 1}, $\phi_k(|n|)$ is positive whence $|n'| = |n| + \phi_k(|n|)$ and $d_h(1,h) \leq |n'|$.  Applying \eqref{phi geq 1} and \eqref{phi geq n} gives $|n| + 1 \leq |n'| \leq 2\phi_k(|n|)$. 

\textit{Case: $\epsilon = 1$ and $n > 0$.}  Following the calculation $$t^n a_k \ = \  a_k {a_k}^{-1} t^n a_k \  =  \   a_k t^n \theta^n({a_k}^{-1}) a_k \ = \  (a_k t) t^{n-1} \theta^{n-1}({a_{k-1}}^{-1}) \ldots \theta^0({a_{k-1}}^{-1})$$ we obtain an instance of conclusion (ii) by taking $n' = n-1$, $h = (a_k t)$ and $$u \ = \  \theta^{n-1}({a_{k-1}}^{-1}) \ldots \theta^0({a_{k-1}}^{-1}).$$  The upper bound on $|n'|$ follows from \eqref{phi geq 1} and \eqref{phi geq n}.  The fact that $t^{n'} \theta^{n-1}({a_{k-1}}^{-1})$ does not lie in $\Lambda$ follows from Lemmas~\ref{lem3} and \ref{lem7}. 
\end{proof}

\begin{lemma} \label{Cancelling at either end}
  Let $\sigma = t^r \theta^a({a_k}^{-1}) \ldots \theta^0({a_k}^{-1}) w \theta^0(a_k) \ldots \theta^b(a_k) t^{-s}$  where $w = w(a_1, \ldots, a_k)$ is freely reduced  and $a, b \geq 0$.  Suppose $\sigma$ represents an element of $H$ but $t^r \theta^a({a_k}^{-1}) \notin \Lambda$ and $t^s \theta^b({a_k}^{-1}) \notin \Lambda$.  Then either  \begin{enumerate}

    \item $w$ has a prefix $\theta^0({a_k}) \ldots \theta^{a-1}({a_k}) a_k$ and suffix ${a_k}^{-1} \theta^{b-1}({a_k}^{-1}) \ldots \theta^0({a_k}^{-1})$, or

    \item $w  = \theta^0(a_k) \ldots \theta^{a-1}(a_k) \theta^{b-1}({a_k}^{-1}) \ldots \theta^0({a_k}^{-1})$.
\end{enumerate}
\end{lemma}

\begin{proof}
  Write $l_1$ for the letter ${a_k}^{-1}$ of the term $\theta^a({a_k}^{-1})$ of $\sigma$ and write $l_2$ for the letter $a_k$ of the term $\theta^b(a_k)$ of $\sigma$.  Lemma~\ref{Key Lemma} implies that, when $\sigma$ is freely reduced, both $l_1$ and $l_2$ cancel.  Let $l'$ be the letter $a_k$ that cancels with $l_1$

  If $l'$ lies in $w$, then $l_2$ must cancel with a letter to the right of $l'$ in $w$, and we have case (i).

  On the other hand, suppose that $l'$ lies in the subword $\theta^0(a_k) \ldots \theta^b(a_k)$.  If $l'$ is distinct from $l_2$, then $l_2$ must cancel with an ${a_k}^{-1}$ lying to the right of $l'$.  But this is a contradiction, since all the occurrences of ${a_k}^{\pm1}$ in $\theta^0(a_k) \ldots \theta^b(a_k)$ are positive.  Thus $l' = l_2$.  Now $\theta^{a-1}({a_k}^{-1}) \ldots \theta^0({a_k}^{-1}) w \theta^0(a_k) \ldots \theta^{b-1}(a_k)$ must be freely trivial and we have case (ii).
\end{proof}

\begin{lemma} \label{ell(v) geq ell(u)}
  Let $w = \theta^0(a_k) \ldots \theta^r(a_k)$ where $r \geq 0$.  Let $l$ be the last $a_k$ appearing in $w$ and partition $w$ as $w = uv$ where $u$ is the prefix of $w$ ending with $l$ and $v$ is the suffix of $w$ coming after $l$.  Then $\ell(u) \geq \ell(v)$.
\end{lemma}

\begin{proof}
  Note that $u = \theta^0(a_k) \ldots \theta^{r-1}(a_k) a_k$ and, by Lemma~\ref{expanding theta}, $v = \theta^0(a_{k-1}) \ldots \theta^{r-1}(a_{k-1})$.  It thus suffices to prove that $\ell(\theta^i(a_k)) \geq \ell(\theta^i(a_{k-1}))$ for $i \geq 0$. But this follows by an easy induction on $k$ from the structures of $\theta^i(a_k)$ and $\theta^i(a_{k-1})$ respectively given by Lemma~\ref{expanding theta}.
\end{proof}

\section{Groups with Ackermannian Dehn functions} \label{Big Dehn}

Recall that $\Gamma_k$ is the HNN extension of $G_k$ over $H_k$ in which the stable letter commutes with all elements of $H_k$:  $$\Gamma_k  \  := \ \langle  \ a_1, \ldots, a_k, t, p \  | \  t^{-1} {a_1} t = a_1, \ t^{-1} {a_i} t = a_i a_{i-1} \  (i > 1), \  [p, a_i t] =1 \  (i > 0) \  \rangle.$$

\begin{prop} \label{quadratic Dehn function}
The group $\Gamma_1$ has Dehn function $\simeq$--equivalent to $n \mapsto n^2$.
\end{prop}

\begin{proof}
Making the substitution $\alpha = a_1 t$ shows that $\Gamma_1$ is a right-angled Artin group with presentation $\langle \, \alpha, t, p \, | \, [t, \alpha], [p, \alpha] \, \rangle$.  It follows that $\Gamma_1$ is $\textup{CAT}(0)$ \cite{CharneyDavis} whence it has Dehn function  $\simeq$-equivalent to $n^2$ by \cite[Proposition~1.6.III.$\Gamma$]{BrH}.
\end{proof}

\begin{prop} \label{prop7}
For all $k \geq 2$, the group $\Gamma_k$ has Dehn function $\simeq$--equivalent to $A_k$.
\end{prop}

\begin{proof}
  Let $k \geq 2$.  The Dehn function of a $\textup{CAT}(0)$ group is either linear or quadratic \cite[Theorem~6.2.1]{Bridson6}, with the linear case occurring precisely when the group is hyperbolic \cite[Theorem 6.1.5]{Bridson6}.  By Theorem~\ref{main}, the group $G_k$ is $\textup{CAT}(0)$.  However, since it contains an embedded copy of $\Z^2$ it is not hyperbolic \cite[Theorem~6.1.10]{Bridson6}.  The Dehn function of $G_k$ is therefore quadratic.  By Theorem~\ref{main}, the distortion function of $H_k$ in $G_k$ is $\simeq$--equivalent to $A_k$.  Plugging these two functions into \cite[Theorem~6.20.III.$\Gamma$]{BrH} gives lower and upper bounds for the Dehn function of $\Gamma_k$ of $\max \{n^2, n A_k(n) \}$ and $n A_k(n)^2$ respectively, up to $\simeq$--equivalence.  By \eqref{A_k2}, this lower bound is equal to $A_k(n)^2$ and the upper bound is at most $A_k(n)^3$.  But \eqref{A_k8} implies that, for any $C \geq 1$, the function $n \mapsto A_k(n)^C$ is $\simeq$--equivalent to $A_k$.
\end{proof}

 \begin{figure}[ht]
\psfrag{p}{$p$}
\psfrag{v}{$v_{2,4}$}
\centerline{\epsfig{file=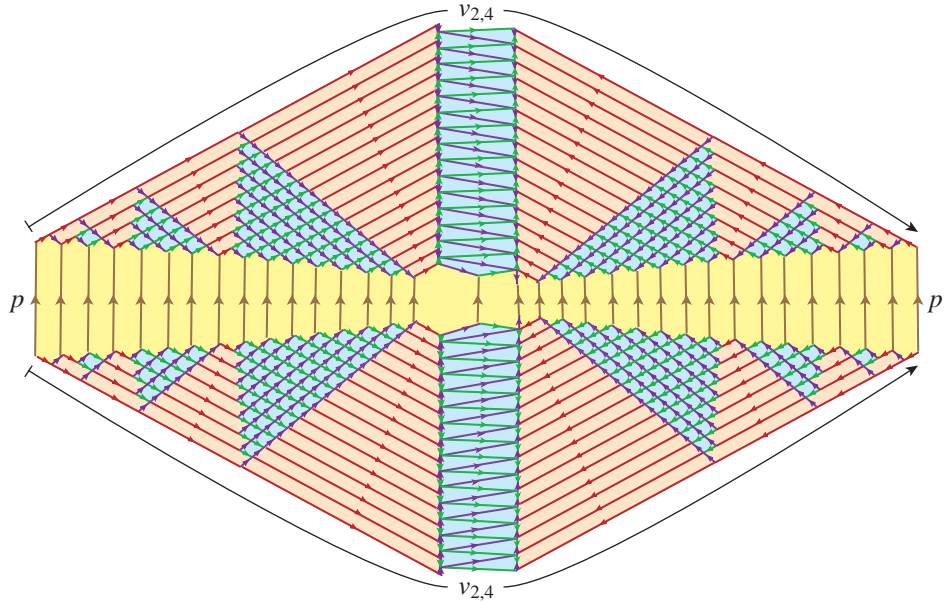}} \caption{A van~Kampen diagram for $[v_{2,4},p]$ --- an example of a word which represents the identity in $\Gamma_k$ but can only be filled by a large area diagram.} \label{huge area diagram}
\end{figure}

The ideas behind \cite[Theorem~6.20.III.$\Gamma$]{BrH} used here are most transparent via the tools of van~Kampen diagrams and corridors.  For example, towards the lower bound, consider the words $$v_{k,n} \ := \   {a_k}^{n} a_2 \  t a_1 \  {a_2}^{-1} {a_k}^{-n}$$  of Section~\ref{DistortionLower}, which equal $$w_{k,n}  \ := \  u_{k,n} \,  (a_2t) \,  (a_1 t) \, ({a_2}t)^{-1} \, {u_{k,n}}^{-1}$$ in $G_k$.
Observe that $[v_{k,n},p]=1$ in $\Gamma_k$ and that in any van~Kampen diagram for  $[v_{k,n},p]$, there must be a $p$--corridor connecting the two boundary edges labelled by $p$.  (Figure~\ref{huge area diagram} is an example of such a diagram when $k=2$ and $n=4$.)  The word on $a_1t, \ldots, a_kt$ written along each  side of this corridor must  
equal $v_{k,n}$ in $G_k$ and so freely equals $w_{k,n}$.  It follows that any van~Kampen diagram for $[v_{k,n},p]$ has area at least the length of $w_{k,n}$, which is $2 \H_k(n) +3$.  So, as the length of $[v_{k,n},p]$ is  $4n+10$,  this leads to a lower bound of $A_k(n)  \simeq \H_k$ on the Dehn function of $G_k$.

\bibliographystyle{plain}
\bibliography{$HOME/Dropbox/Bibliographies/bibli}

\ni
\small{\textsc{Will Dison} \rule{0mm}{6mm} \\
Department of Mathematics,
University Walk, Bristol, BS8 1TW, UK \\ \texttt{w.dison@bristol.ac.uk}, \
\href{http://www.maths.bris.ac.uk/~mawjd/}{http://www.maths.bris.ac.uk/$\sim$mawjd/}

\ni  \textsc{Timothy R.\ Riley} \rule{0mm}{6mm} \\
Department of Mathematics, 310 Malott Hall,  Cornell University, Ithaca, NY 14853, USA \\ \texttt{tim.riley@math.cornell.edu}, \
\href{http://www.math.cornell.edu/~riley/}{http://www.math.cornell.edu/$\sim$riley/}}
\end{document}